\documentclass[a4paper,11pt]{article}

\usepackage[pdftex]{hyperref}
\usepackage{stmaryrd}
\usepackage[normalem]{ulem}
\usepackage{graphicx,xcolor}
\usepackage{amsfonts}
\usepackage{amsmath,amssymb,mathrsfs,amsthm}
\usepackage[all]{xy} 
\usepackage{color}
\usepackage{times}
\usepackage{enumerate,vmargin}
\usepackage{verbatim}
\usepackage{paralist}
\usepackage[labelfont=bf,size=small,font=it]{caption}
\usepackage[numbers,sort&compress]{natbib}

\setpapersize{A4}
\setmarginsrb{2.5cm}{2cm}{2.5cm}{2cm}{.4cm}{4mm}{0cm}{7.5mm}

\hypersetup{
    unicode      = false,     % non-Latin characters in Acrobat’s bookmarks
    pdftoolbar   = true,      % show Acrobat’s toolbar?
    pdfmenubar   = true,      % show Acrobat’s menu?
    pdffitwindow = true,      % page fit to window when opened
    pdfnewwindow = true,      % links in new window
    colorlinks   = true,      % false: boxed links; true: colored links
    linkcolor    = blue,      % color of internal links
    citecolor    = blue,      % color of links to bibliography
    filecolor    = blue,      % color of file links
    urlcolor     = blue       % color of external links
}

\theoremstyle{plain} 
\newtheorem{lem}{Lemma}
\newtheorem{thm}[lem]{Theorem}
\newtheorem{prop}[lem]{Proposition}
\newtheorem{cor}[lem]{Corollary}

\newcommand{\black}{\color{black}}

\definecolor{mm}{rgb}{0,.6,0.4}

\theoremstyle{definition}

\theoremstyle{remark}

\newtheorem{alg}{Algorithm}

\newcommand{\bP}{\mathbf P}
\newcommand{\bE}{\mathbf E}

\newcommand{\llb}{\llbracket}
\newcommand{\rrb}{\rrbracket}

\newcommand{\ep}{\epsilon}

\newcommand{\Sk}{\operatorname{Sk}}
\newcommand{\Lf}{\operatorname{Lf}}
\newcommand{\Br}{\operatorname{Br}}

\newcommand{\Span}{\operatorname{Span}}
\newcommand{\cut}{\operatorname{cut}}
\newcommand{\shuff}{\operatorname{shuff}}

\newcommand{\SSub}{\operatorname{S}}

\newcommand{\bv}{\mathbf v}

\newcommand{\bx}{\mathbf x}

\newcommand{\bm}{\mathbf m}

\newcommand{\bcT}{\boldsymbol{\mathcal T}}

\newcommand{\bU}{\mathbf U}
\newcommand{\bV}{\mathbf V}
\newcommand{\bX}{\mathbf X}

\newcommand{\cF}{\mathcal F}
\newcommand{\cG}{\mathcal G}
\newcommand{\cH}{\mathcal H}
\newcommand{\cI}{\mathcal I}

\newcommand{\cM}{\mathcal M}
\newcommand{\cN}{\mathcal N}

\newcommand{\cP}{\mathcal P}

\newcommand{\cT}{\mathcal T}

\newcommand{\cS}{\mathcal S}

\newcommand{\rR}{\mathrm R}

\newcommand{\m}{\mathrm m}

\newcommand{\D}{\mathrm{Decomp}}

\newcommand{\sI}{\mathscr I}

\newcommand{\bbR}{\mathbb R}
\newcommand{\bbT}{\mathbb T}

\newcommand{\bbP}{\mathbb P}
\newcommand{\bbE}{\mathbb E}

\newcommand{\bbM}{\mathbb M}
\newcommand{\bbN}{\mathbb N}
\newcommand{\N}{\mathbb N}

\newcommand{\eqd}{\stackrel{d}=}

\newcommand{\Ec}[1]{{\mathbb E}[#1]}

\newcommand{\br}{\mathrm{br}}
\newcommand{\Rt}{\mathrm{root}}

\graphicspath{{.}{figures/}}
\title{\bf Reversing the cut tree of the\\ Brownian continuum random tree}
\author{Nicolas Broutin
\thanks{INRIA de Paris, 2 rue Simone Iff, CS 42112, 75589 Paris Cedex 12 - France.
Email: nicolas.broutin@inria.fr} 
\and Minmin Wang
\thanks{Universit\'e Pierre et Marie Curie, 4 place Jussieu, 75005 Paris - France.
Email: wangminmin03@gmail.com }
}
\date{\today}

\begin{document}

\maketitle

\begin{abstract}
Consider the Aldous--Pitman fragmentation process \cite{aldcoal} of a Brownian continuum random tree $\cT^{\mathrm{br}}$. The associated cut tree $\cut(\cT^{\mathrm{br}})$, introduced by Bertoin and Miermont \cite{Bert12}, is defined in a measurable way from the fragmentation process, describing the genealogy of the fragmentation, and is itself distributed as a Brownian CRT. In this work, we introduce a \emph{shuffle} transform, which can be considered as the reverse of the map taking $\cT^{\br}$ to $\cut(\cT^{\br})$.

\smallskip


\smallskip

\end{abstract}
{\small

%\tableofcontents
}
\section{Introduction}\label{sec:intro}

\subsection{Motivation and literature}
Let $\mathcal T^{\br}$ be Aldous' Brownian continuum random tree (CRT). We consider the fragmentation process introduced by Aldous \& Pitman \cite{aldcoal}: informally, the process describes the time evolution of the masses of the connected components of a forest $\cF_t$, $t\ge 0$, where $\cF_t$ results from a logging of $\cT^{\br}$ with cuts falling uniformly per unit of time and length in $\cT^{\br}$. 
The Aldous--Pitman fragmentation  is an instance of a self-similar fragmentation such as studied in Bertoin's book \cite{bertoinfrag}. 
There is a natural genealogical structure associated with the fragmentation process, and it is as a representation of this genealogy that Bertoin and Miermont  \cite{Bert12} constructed the so-called cut tree of $\cT^{\br}$, hereafter denoted by $\cut(\cT^{\br})$. A rather remarkable fact is that $\cut(\cT^{\br})$ is itself also distributed as the Brownian CRT. In this work, we are interested in defining the reverse of the map $\cT^{\br}\mapsto \cut(\cT^{\br})$. This has been motivated by a seemingly natural question: given the cut tree, can one reconstruct the initial tree? We will see that the cut tree does not contain all the information necessary for such a reconstruction; this observation leads us then to introduce a reverse transform. 
However, giving a proper meaning to the reverse transform
requires some explanation, which we postpone to Sections~\ref{sec: intro_discret} and \ref{sec: intro_plan}. For the time being, we provide some background on cut trees, which can be traced back to some work in combinatorics dating from the seventies. 

\paragraph{Random cutting of trees. }
The idea of cut trees is closely related to random cutting of trees, a subject initiated by Meir and Moon \cite{meir69} and that has since then been largely studied. The initial question concerns discrete trees, and we present here 
a version of the random cutting problem where cuts happen at nodes (one can also define a version where cuts happen at edges): take a rooted tree (random or not) on a finite vertex set; at each step, sample a node uniformly at random and remove it along with all the edges adjacent to it (the removed node is then referred to as a \emph{cut}); this disconnects the tree into connected components (maybe one, if we picked a leaf, for instance); discard the components that are now disconnected from the root; keep going until the root is finally picked. The main questions addressed by Meir and Moon and many subsequent researchers concern mostly the number of cuts that are needed for the process to terminate. This problem has been considered for a number of classical models of deterministic and random trees, including random binary search trees \cite{Holmgren2010a,Holmgren2011c}, random recursive trees \cite{IkMo2007,DrIkMoRo2009,
Bertoin2013b,BaBe2014a} and Galton--Watson trees conditioned on the total progeny 
\cite{Janson06,ABH10,FiKaPa2006,Panholzer2006, Bert11, Bert12}. 

The CRT being the scaling limits of Galton--Watson trees with finite-variance offspring distribution
\cite{aldcrt3}, the case of Galton--Watson trees is the most related to our matters, and we now focus on that 
case: let $T_n$ be a Galton--Watson tree conditioned to have $n$ nodes, and whose offspring distribution has variance $\sigma^2<\infty$; denote by $N_n$ the number of cuts until the root $V$ is picked in the above process. Then Janson \cite{Janson06} showed by moment calculations that, as $n\to\infty$, 
$N_n/(\sigma\sqrt{n})$ converges in distribution to the Rayleigh distribution. Incidentally, the Rayleigh distribution is also the limit law of $(\sigma H_n/\sqrt{n})_{ n\ge 1}$, where $H_n$ denotes the height (distance to the root) of a randomly picked node in $T_n$. Thus, Janson's result can be rephrased as follows: the limit distribution of $(N_n/(\sigma\sqrt{n}))_{n\ge 1}$ coincides with that of $(\sigma H_n/\sqrt{n})_{n\ge 1}$. 
It turns out that an even stronger statement holds true   
in the case that $T_n$ is a uniform labelled tree of $n$ nodes (this is equivalent to take the offspring distribution to be  Poisson(1), and is sometimes referred to as the Cayley tree): we have that $N_n$ and $H_n+1$ actually have exactly the same distribution. This result is due to Addario-Berry, Broutin and Holmgren \cite{ABH10}, and relies on the following bijective method: one can construct another tree $\cut(T_n; V)$ (on the same vertex set as $T_n$) which encodes the isolation of $V$ by the successive cuts (see Section \ref{sec: intro_discret} for the details) such that (1) the node $V$ lies at distance $N_n-1$ from the root, and (2) $\cut(T_n; V)$ has the same distribution as $T_n$, while (3) $V$ is a uniform node in $\cut(T_n; V)$. The above distributional identity then follows.  We call $\cut(T_n; V)$ the $1$-partial cut tree, since 
it keeps track of the way one node (here $V$) was isolated. More generally, Addario-Berry, Broutin and Holmgren \cite{ABH10} have considered cutting procedures resulting in the isolation of $k$ nodes and introduced the corresponding $k$-partial cut trees. For these cutting procedures, one only discards the portions of the tree that do not contain any of the $k$ marked nodes to be isolated. Moreover, by first taking a uniform permutation of the vertex set, we can define simultaneously all the $k$-isolation processes, so that letting $k\to \infty$ we obtain the (complete) cut tree $\cut(T_n)$ of $T_n$, whose graph distance encodes the number of cuts required to isolate every single one of the $n$ nodes. In this case, since all the nodes are marked, no portion of the tree is ever discarded, and the tree 
$\cut(T_n)$ actually encodes the genealogy of a discrete fragmentation of the tree
%in which at each time step a random node is sampled (%It is easily seen that proceeding simultaneously in all the connected components, or 
%sampling one node at a time does not influence the structure of cuts.
(we refer to Section \ref{sec: intro_discret} for details). A similar notion appears in Bertoin \cite{Bert11} and Bertoin \& Miermont \cite{Bert12}, where they define a (different) cut tree $\widehat{\cut}(T_n)$ for $T_n$ 
directly as the genealogy tree of the discrete fragmentation process induced by the cutting of $T_n$.

\paragraph{Random cutting of continuum trees and fragmentation processes. }
More recently, such cutting processes have been considered for the Brownian CRT \cite{ABH10, AbDe11, Bert12, Bert11}. The cutting on the Brownian CRT is of course closely related to the Aldous--Pitman fragmentation mentioned in the first paragraph. 
Moreover, Bertoin and Miermont \cite{Bert12} proved that if $T_n$ is a Galton--Watson tree with a finite variance offspring distribution and conditioned to have $n$ nodes, then the pair $(T_n, \widehat{\cut}(T_n))$, after suitable scaling in the graph distance, converges in distribution in the sense of Gromov--Prokhorov, to a pair $(\cT^{\br}, \cut(\cT^{\br}))$ of continuum random trees; furthermore the tree $\cut(\cT^\br)$ can be defined directly from the fragmentation process of $\cT$ and indeed encodes its genealogy. A similar result holds for $(T_n, \cut(T_n))$ in the case where $T_n$ is a uniform Cayley tree; see \cite{BW}.

Let us also mention that cut trees have been introduced for other models of continuum random trees, including L\'evy trees under excursion measures \cite{AbDe12}, stable trees conditioned on the total masses \cite{Dieuleveut2013a}, and inhomogeneous continuum random trees \cite{BW}.

\subsection{Reversing the cut trees of Cayley trees}
\label{sec: intro_discret}

Although our main concern is the case of the continuum tree, we think it will be helpful to explain here the question we address and our approach to its solution in the setting of discrete trees. The case of Cayley 
trees, for which the question has been studied in \cite{ABH10} for partial reversals and then in \cite{BW} for the complete reversal, is especially adapted to our presentation since many of the correspondences are then exact. We refer to these two papers for proofs and further details. 

Throughout this part, let $\bbT_n$ denote the set of rooted labelled trees on the vertex set $[n]\!:=\{1, 2, \dots, n\}$ and let $T\in \bbT_n$. For $u, v\in [n]$, we write $\llb u, v\rrb$ for the set of vertices that lie on the 
shortest path joining $u$ to $v$ in $T$. 
Let $X(1), X(2), \dots, X(n)$ be a uniform permutation of the vertex set. We will use the sequence $(X(i))_{1\le i\le n}$ to define various isolation processes on $T$. 

\begin{figure}[tbp]
    \centering
    \includegraphics[width=0.85\textwidth]{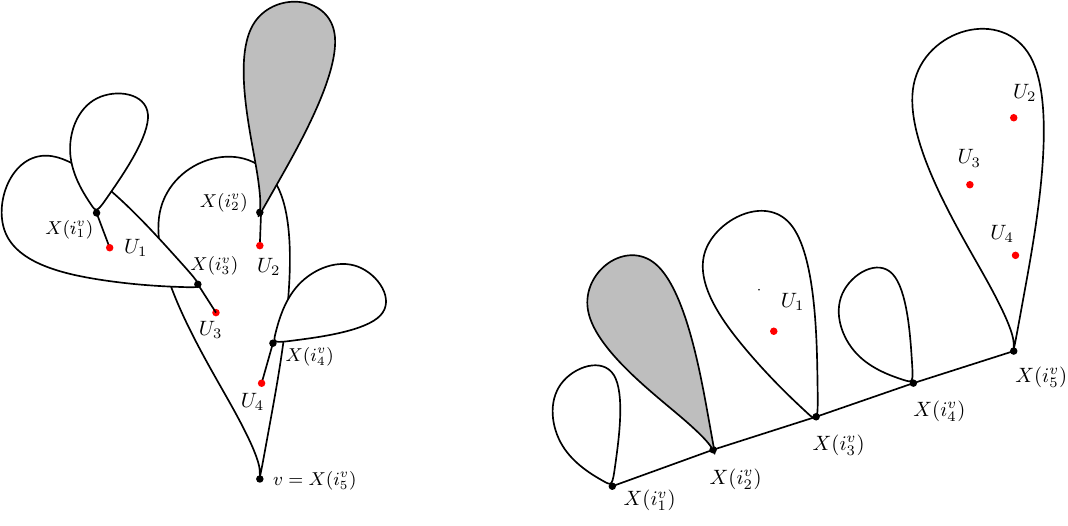}
    \caption{An example of one-node isolation and the associated $1$-partial cut tree. On the left, the initial tree $T$ and the sequence of cuts $X(i^v_1), X(i^v_2), \dots, X(i^v_5)$, which are responsible for  the isolation of $v=X(i^v_5)$. On the right, the corresponding cut tree. Note that due to the cut at $X(i^v_m)$, a small subtree of $T$ is discarded; it then appears in $\cut(T; v)$ as the subtree above the path $\llb X(i^v_1), X(i^v_5)\rrb$.   }
    \label{fig: 1-cut}
\end{figure}

\paragraph{One-node isolation and the $1$-partial cut tree. }
Let $v$ be any node of $T$ and consider the following isolation process of $v$. 
Let $F^v_0=T$ and for $1\le m\le n$, let $F^v_m$ be the connected component of $T\setminus\{X(1), \dots, X(m)\}$ containing $v$, with the convention that $F^v_{m'}=\varnothing$ for all $m'\ge m$ if $X(m)=v$. We say that $X(m)$ is a \emph{cut} in this process if $X(m)\in F^v_{m-1}$. Namely, the cuts are those elements of $[n]$ whose removal have reduced the size of the current connected component of $v$. Let $\mathbf X(v):=\{X(i_m^v): 1\le m\le N(v), i_1^v <i_2^v<\cdots< i^v_{N(v)}\}$ be the sequence formed by the successive cuts. Observe in particular that $X(i^v_{N(v)})=v$. For $1\le m\le N(v)-1$, we let $U_m$ be the unique neighbor of $X(i_m^v)$ which belongs to $\llb X(i_m^v), v\rrb$. The following algorithm returns a tree denoted by $G(T, v, \mathbf X(v))$ as a function of $T, v$ and $\mathbf X(v)$. 

\begin{alg}\label{alg: 1-cut}
{\bf Construction of the $1$-partial cut tree.}  Apply the following transformations to the tree $T$:
\begin{compactenum}[--]
\item
for $1\le m\le N(v)-1$, remove the edge $\{X(i^v_m), U_m\}$;
\item
for $1\le m\le N(v)-1$, add the edge $\{X(i_m^v), X(i_{m+1}^v)\}$;
\item
declare $X(i_1^v)$ as the root. 
\end{compactenum}
\end{alg}
Denote by $\cut(T; v)=G(T, v, \mathbf X(v))$ the graph thereby obtained (see also Figure \ref{fig: 1-cut}). Then we have $G(T, v, \mathbf X(v))\in \bbT_n$ and it contains a path consisting of the sequence $\bX(v)$. Moreover, if we remove all the edges in  that path, then each $X(i^v_k)$ remains connected to the subgraph $F_{i^v_k-1}\setminus F_{i^v_k}$, namely the part discarded at step $k$ because of the cut at $X(i^v_k)$.   

\paragraph{$k$-node isolation and the $k$-partial cut tree. } 
The above isolation process can be generalized to the case of multiple nodes. Let $\bv=\{v_1, v_2, \dots, v_k\}$ be $k\ge 1$ vertices of $T$ (not necessarily distinct). For each $v_j$ and $m\ge 0$, let $F^{v_j}_m$ denote the connected component of $T\setminus\{X(i): i\le m\}$ containing $v_j$; then the sequence of cuts responsible for the isolation of $v_j$, namely, those $X_m$ satisfying $X_m\in F^{v_j}_{m-1}$, is denoted by $\{X(i^{v_j}_m): 1\le m\le N(v_j), i^{v_j}_1<i^{v_j}_2<\cdots<i^{v_j}_{N(v_j)}\}$. 

To define the associated partial cut tree, we adopt a recursive approach. For $2\le j\le k$, let $m_j=\max\{m: v_j\in \cup_{1\le j'\le j-1}F^{v_{j'}}_m\}$. Then set $\mathbf X(v_1)=\{X(i^{v_1}_m): 1\le m\le N(v_1)\}$ and for $2\le j\le k$, set $\mathbf X(v_j)=\{X(i^{v_j}_m): i^{v_j}_m> m_j, m\le N(v_j)\}$. 

\begin{alg}\label{alg: k-cut}
{\bf Construction of the $k$-partial cut tree.}   Set $G_1=G(T, v_1, \mathbf X(v_1))$, and for $2\le j\le k$, do the following: if $\mathbf X(v_j)=\varnothing$, set $G_j=G_{j-1}$; otherwise, 
\begin{compactenum}[--]
\item
locate in $G_{j-1}$ the connected component of $G_{j-1}\setminus(\mathbf X(v_1)\cup\cdots\cup\mathbf X(v_{j-1}))$ which contains $v_j$, denote it by $T_j$; let $w_j$ be the node of $\mathbf X(v_1) \cup \cdots \cup \mathbf X(v_{j-1})$ that is closest to $v_j$ in $G_{j-1}$;
\item
replace in $G_{j-1}$ the subgraph $T_j$ by $G(T_j, v_j, \mathbf X(v_j))$: {\black to do so, remove the edge between $w_j$ and $T_j$ and add the edge $\{w_j, X(m_j+1)\}$ }; let $G_j$ be the graph obtained, still rooted at $X(i_1^{v_1})$.
\end{compactenum}
\end{alg}
Denote by $\cut(T; \mathbf v)=G_k$ the graph thereby obtained. Then $\cut(T; \mathbf v)\in \bbT_n$ and the 
subset of nodes $\bX(v_1) \cup \cdots \cup \bX(v_k)$ forms a subtree that contains the root $X(i_1^{v_1})$ (See Figure~\ref{fig:rec_partial_cut}).

\begin{figure}[tbp]
    \centering
    \includegraphics[width=0.9\textwidth]{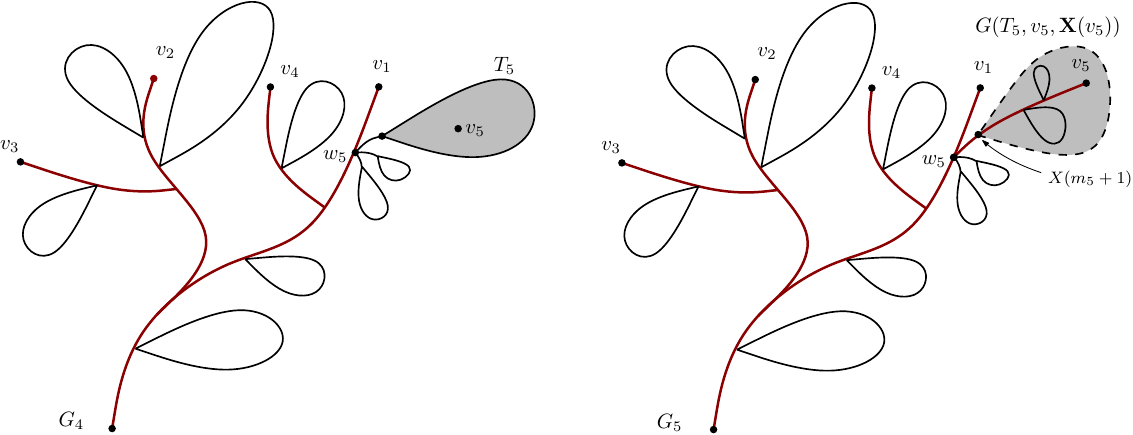}
    \caption{The recursive construction of partial cut trees. To obtain $G_5$ from $G_4$, we simply replace the subtree $T_5$ (the grey part on the left) by the $1$-partial cut tree $G(T_5, v_5, \bX(v_5))$ (the grey one on the right). The red paths on the left consist of $\bX(v_1)\cup\cdots\cup\bX(v_4)$; the ones on the right consist of  $\bX(v_1)\cup\cdots\cup\bX(v_5)$. }
    \label{fig:rec_partial_cut}
\end{figure}

\paragraph{The complete cut tree. }
Let $(V_i)_{i\ge 1}$ be a sequence of independent and uniform nodes of $T$ and write $\mathbf V_k=\{V_1, \dots, V_k\}$, $k\ge 1$. Observe that almost surely the sequence $\{\cut(T; \mathbf V_k): k\ge 1\}$ becomes stationary after some $k$. Denote by $\cut(T)$ the limit of this sequence; it has the following remarkable properties. First, for each vertex $v$, the path in $\cut(T)$ leading to $v$ consists of  precisely the cuts responsible for the isolation of $v$. Note there is at most one tree in $\bbT_n$ satisfying this property. We conclude that $\cut(T)$ does not depend on $(V_i)_{i\ge 1}$.  
Second, $\cut(T)$ is uniformly distributed in $\bbT_n$. %Third, $\cut(T)$ does not depend on the sequence $(V_i)_{i\ge 1}$,{\red since $\{\cut(T, \mathbf V_k): k\ge 1\}$ is eventually stationary.

Note that we cannot recover $T$ from $\cut(T)$. To explain this, let us first introduce the following notation. 
For $T\in \bbT_n$ and two vertices $u\ne v$, let $\SSub(T, u\,|\,v)$ denote the 
connected component of $T\setminus\{u\}$ which contains $v$. 
%\fnic{I have a problem with this notation, since the roles of $u$ and $v$ appear rather symmetric \comm{Is $S_u(T, v)$ better?} Using subscript is not really better since the names for the nodes are actually huge in general. What about $\SSub(T,u\,|\,v)$ ?.}
If $X(i^{_{V_j}}_m)$ is a cut in the isolation of $V_j$ as defined above, denote by $U^{j}_m$ the neighbor of $X(i^{_{V_j}}_m)$ that belongs to $\llb V_j, X(i^{_{V_j}}_m)\rrb$. Then one can show that $U^{j}_m$ is uniformly random in $\SSub(\cut(T), X(i^{_{V_j}}_m)\,|\,V_j)$ (see also Figure \ref{fig: 1-cut}). 
In particular, this means that the information concerning the whereabouts of $U^j_m$ is partially lost in $\cut(T)$; therefore we cannot know the initial tree just from its cut tree. 
On the other hand, we know the distribution of $(U^j_m)$ conditional on $\cut(T)$. 
Relying on this, we can ``resample'' the lost information, namely, take a random collection $(\hat U^j_m)$ 
according to the distribution of $(U^j_m)$ conditional on $\cut(T)$; we then 
``reconstruct'' $T$ from $\cut(T)$ by assuming $(\hat U^j_m)$ are the actual $(U^j_m)$. Of course, we will not obtain from this procedure the actual $T$ with probability $1$, but instead a random tree which has 
the distribution of $T$ given $\cut(T)$. 
This is the basic idea of our reverse transform for the mapping $T\mapsto \cut(T)$. The following paragraphs explain how to proceed in the case of Cayley trees. 

\begin{figure}[htb]
    \centering
    \includegraphics[scale=.9]{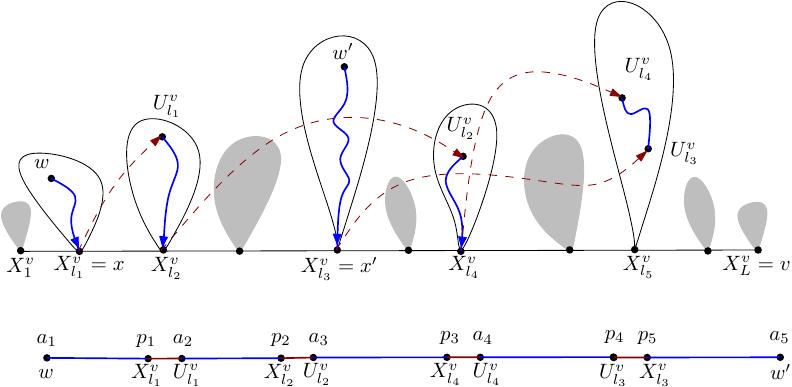}
    \caption{\label{fig:path} \black
    The $1$-partial shuffle transform as viewed from two points $w, w'$. In the upper line is illustrated the tree $T$; the straight line represents the path $\{X^v_l: 1\le l\le L\}$ joining the root to $v$.
Removing the edges $\{X^v_{l}, X^v_{l+1}\}, 1\le l\le L-1$, $T$ falls into $L$ subtrees (blocks in white and in gray in above).   
For $l<L$, the point $U^v_l$ is located in one of these subtrees right to $X^v_l$. By adding an edge between $X^v_l$ and $U^v_l$ (represented by the red arrows in above), we shuffle these subtrees  and make them into a new graph $H$. \\
Observe that the path of $H$ joining $w$ to $w'$ only crosses a sub-collection of these subtrees. In the example above, they correspond to the white subtrees.  
In the line below is depicted this path. 
Observe that it consists of $m$ blue segments and $m-1$ red segments, where $m$ is the number of subtrees it crosses ($m=5$ in the example above). 
Each of the red segments has length one; they are the edges added between $X^v_l$ and $U^v_l$. Each of the blue segments was contained in a white subtree; so that its length is the same in $H$ as in $T$. This explains Equation \eqref{eq: 1-rev}. The pair $(a_i, p_i)$ consists of the endpoints of the $i$-th blue segment. Their relative positions (which one is on the left) are in fact unimportant. Here, we have followed a choice convenient for generalization. In this example, their respective values in $(X^v_l)_{1\le l\le L}\cup\bU(v)\cup\{w, w'\}$ are given by the labels below. }
\end{figure}

\paragraph{Reversing the $1$-partial cut tree transform. } 
Let $v$ be a vertex of $T$. Suppose that $X^v_1, \dots, X^v_{L}=v$ is the sequence of vertices along the path of $T$ from the root to $v$. For $1\le l\le L-1$, sample a random vertex $U^v_i$ uniformly in $\SSub(T, X^v_l \,|\,v)$. Write $\bU(v)=\{U^v_i: 1\le l\le L-1\}$. 
Sample a uniform vertex $\rho_1$ of $T$. The following algorithm returns a tree $H(T, v, \bU(v), \rho_1)$ as a function of $T$, $v$, $\bU(v)$ and $\rho_1$. 

\begin{alg}\label{alg: 1-rev}
{\bf $1$-partial shuffle transform.}  Take $T$ and do the following: 
\begin{compactenum}[--]
\item
for $1\le l\le L-1$, remove the edge $\{X^v_{l}, X^v_{l+1}\}$;
\item
for $1\le l\le L-1$, add the edge $\{X^v_{l}, U^v_{l}\}$; 
\item
declare $\rho_1$ as the root. 
\end{compactenum}
\end{alg}

Denote by $H=H(T, v, \bU(v), \rho_1)$ the graph thereby obtained, which turns out to be an element of $\bbT_n$. Observe that the above algorithm is the exact reverse of Algorithm \ref{alg: 1-cut}. Let us also make the following observation, which is useful for the generalization to continuum trees. Let $w, w'$ be two distinct vertices of $T$ and let $x$ (resp.~$x'$) be the vertex among $\{X^v_l: 1\le l\le L\}$ which is closest to $w$ (resp.~$w'$). To simplify the discussion, suppose that $x, x', v$ are all distinct; the path in $H$ joining $w$ to $w'$ contains a sub-collection of $(X^v_l)_{1\le l\le L}\cup\bU(v)\cup\{w, w'\}$; we then match the elements of this sub-collection into pairs: $\{(a_i, p_i): 1\le i\le m\}$, so that for each $i$, $a_i$ and $p_i$ are in the same connected component after removing the edges $\{X^v_{l}, X^v_{l+1}\}, 1\le l\le L-1$. 
See Fig.~\ref{fig:path} for an example. If we write respectively $d$ for the graph distance in $T$ and $d_{H}$ for that in $H$, then we have (see also Fig.~\ref{fig:path})
\begin{equation}\label{eq: 1-rev}
d_{H}(w, w')=\sum_{1\le i\le m} d(a_i, p_i)+m-1.
\end{equation}

\paragraph{Reversing the cut tree transform. } 
The above partial shuffle transform can be extended to several vertices. To explain this, let $(V_i)_{i\ge 1}$ be a sequence of independent uniform vertices of $T$ and write $\mathbf V_k=\{V_1, \dots, V_k\}$, $k\ge 1$. 
Let $r$ denote the root of $T$.
Let $\Span(T, \mathbf V_k)$ be the smallest connected subgraph of $T$ containing $\mathbf V_k \cup \{r\}$. 
Let $T_1=T$ and for $k\ge 2$, if $V_k\in \Span(T, \mathbf V_{k-1})$, set $T_k=\bU_k=\rho_k=\varnothing$; otherwise let $T_k$ be the connected component of $T\setminus \Span(T, \mathbf V_{k-1})$ which contains $V_k$.
For each $k\ge 1$ and $T_k\ne\varnothing$, sample a uniform vertex $\rho_k$ in $T_k$;
 let $\{X^{k}_1, \dots, X^k_{L_k}\}=\llb r, V_k\rrb \cap T_k$, with the $X^k_i$ ordered by decreasing 
 distances to $V_k$; then, for $1\le l\le L_k-1$, sample a uniform vertex $U^{k}_l$ among $\SSub(T, X^k_l \,|\,V_k)$; let $\bU_k=\{U^k_l: 1\le l\le L_k-1\}$. Then, the following algorithm returns a tree $H(T,  \mathbf V_k, (\mathbf U_j)_{1\le j\le k}, (\rho_j)_{1\le j\le k})$ as a function of $T, \mathbf V_k, (\mathbf U_j)_{1\le j\le k}$ and $(\rho_j)_{1\le j\le k}$. 
 
\begin{alg}\label{alg: k-rev}
{\bf $k$-partial shuffle transform.}   Take $T$ and do the following: 
\begin{compactenum}[--]
\item
Remove the edges of $\Span(T, \mathbf V_k)$. 
\item
For $1\le j\le k$ and $T_k\ne\varnothing$, do the following:
\footnote{The formulation here is slightly different from the one in \cite{BW}, which is stated as follows: replace each edge $(x, w)$ of $\Span(T, \mathbf V_k)$ by $(x, U_w)$, where $x$ is the one closer to the root than $w$ and $U_w$ is a uniform node sampled  in the subtree of $T$ above $(x, w)$, and then root the obtained tree at $\rho_1$. 
It is not difficult to see that this gives the same transform as Algorithm \ref{alg: k-rev}. }
\begin{compactenum}[--]
\item
For $1\le l\le L_j-1$, add an edge $\{X^j_l, U^j_l\}$ for $1\le l\le L_j-1$;
\item
add an edge $\{\rho_k, w\}$, where $w$ is the vertex of $\Span(T, \mathbf V_k)$ closest to $T_k$;\footnote{Observe that this step amounts to rooting the subgraph replacing $T_k$ at $\rho_k$.} 
\end{compactenum}
\item
Declare $\rho_1$ as the root. 
\end{compactenum}
\end{alg}
Denote by $H_k=H(T, \mathbf V_k, (\bU_j)_{1\le j\le k}, (\rho_j)_{1\le j\le k})$ the graph produced, which is an element of $\bbT_n$.  Alternatively, $H_k$ can be obtained in the following recursive way, which can be seen as the dual of Algorithm \ref{alg: k-cut}. (See also  Figure~\ref{fig:rec_partial_cut}, from right to left.)
 Let $k\ge 2$. 
\begin{itemize}
\item
If $T_k=\varnothing$, then $H_k=H_{k-1}$.  
\item 
Otherwise, replace in $T$ the subgraph $T_k$ by $H(T_k, \mathbf U_k, \rho_k)$ and denote by $R_k$ the resulting graph. Then we have $H_k=H(R_k, \mathbf V_{k-1}, (\mathbf U_j)_{1\le j\le k-1}, (\rho_1)_{1\le j\le k-1})$. 
\end{itemize}

Note that we have $\Span(R_k, \mathbf V_{k-1})=\Span(T, \mathbf V_{k-1})$ (with obvious notation); therefore Algorithm~\ref{alg: k-rev} can still be applied to define  $H(R_k, \mathbf V_{k-1}, (\mathbf U_j)_{1\le j\le k-1}, (\rho_1)_{1\le j\le k-1})$. 
It is not difficult to see that the sequence $(H_k)_{k\ge 1}$ is eventually stationary, and denote by $\shuff(T)$ its limit. If $T$ is distributed uniformly in $\bbT_n$, then we have 
\[
\big(T, \shuff(T)\big) \eqd \big(\cut(T), T\big).
\]
Thanks to this identity in distribution, we can consider the shuffle transform as the reverse of the transform $T\mapsto \cut(T)$. 
Let us also keep in mind that in the definition of $\shuff(T)$, we have sampled the following random variables for each $k\ge 1$:     
i) a uniform vertex $\rho_k$ in $T_k$ which ``serves'' as the initial root of the subtree containing $V_k$; ii) for each vertex $X^k_l$ on the path of $T_k$ from its root to $V_k$, a uniform vertex $U^k_l$ in $\SSub(T, X^k_l \,|\, V_k)$ which ``serves'' as a neighbor of $X^k_l$ in the initial tree.

\subsection{An overview of the paper}
\label{sec: intro_plan}

The aim of the paper is to introduce the \emph{shuffle transform} for real trees, and more specifically for the Brownian continuum random tree: for a rooted real tree $(\cT, d, \rho)$ equipped with a finite measure $\mu$, we define a random symmetric matrix $\Gamma_\infty=(\gamma_\infty(i, j))_{1\le i, j<\infty}$ whose entries take values in $\bbR_+\cup\{\infty\}$, such that when $(\cT, d, \rho)$ is distributed according to the distribution of the Brownian CRT, which we denote by $\bP$, almost surely $\Gamma_\infty$ is well defined and characterizes a random measured and  rooted real tree $\shuff(\cT)=(\cH, d_{\cH}, \mu_{\cH}, \rho_{\cH})$; moreover, the law under $\bP$ of the pair $(\cH, \cT)$ seen as measured real trees is the same as that of $(\cT, \cut(\cT))$ under $\bP$. 

This shuffle transform can be viewed as an extension to the Brownian CRT of the construction in Section \ref{sec: intro_discret} for Cayley trees. 
As we have already mentioned, for a discrete tree $T$, the associated cut tree does not contain all the information necessary to reconstruct $T$. This is also true for continuum trees. Therefore, in defining the shuffle transform, we begin by sampling a collection of random points of $\cT$ referred to as the \emph{marks} which replace the lost information in the cut tree. We then construct a real tree 
which corresponds to the initial tree,  if the ``resampled'' information corresponds to the actual one. 
Recall that for discrete trees, we have used an approximation procedure: we first consider the reversals of the partial cut trees and then define the shuffle transform as the limit of  the partial reversals. The advantage of this approach lies in that the partial cut trees retain some unmodified portions of the initial tree, therefore their reversals are easier to handle.  This is even more important in the continuum tree setting.  In that case, the reconstruction consists in roughly three steps:
\begin{itemize}[--]
\item
Define the $1$-partial shuffle transform for the CRT. This has been done in \cite{BW}. 
Let us briefly explain the idea there. Algorithm \ref{alg: 1-rev} does not generalize directly to the CRT, but Equation \eqref{eq: 1-rev} does. Indeed, if $\eta, \eta'$ are two independent uniform points of the Brownian CRT, then during the one-node isolation process, there is only a finite number of cuts falling on the geodesic of the CRT which joins $\eta$ to $\eta'$ (see Lemma \ref{lem: key} below for a precise statement). This suggests that the number of summands in \eqref{eq: 1-rev} remains bounded for the CRTs and we can ``recover'' the distance between $\eta$ and $\eta'$ by a generalization of \eqref{eq: 1-rev}.
\item
Define the $k$-partial shuffle transform for the CRT, for $k\ge 2$. We use a recursive procedure,  similar to the one for Cayley trees. In particular, the recursive construction provides a natural coupling between the different partial reversals, which is convenient for the proof of convergence in the next step. 
\item
The convergence of $k$-partial shuffle transforms as $k\to \infty$. Contrary to the case of discrete trees, for CRTs, this convergence is non trivial, and a significant part of the paper is devoted to its proof. Note that our proof relies crucially on some specific properties of the Brownian CRT, especially the scaling property. 
\end{itemize}

\medskip

The rest of the paper is organized as follows. In Section \ref{sec: notation}, we introduce the necessary notation and recall from \cite{Bert12} the definition of the cut tree of the Brownian CRT. We also collect some results from \cite{BW} that are useful later on. In Section \ref{sec: shuff}, we give the formal definition of  the shuffle transform, which is defined as the limit of partial reversal transforms and state our main result (Theorem \ref{thm: cv_gamma}).  The proof for the convergence of the partial reversals is found in Section \ref{sec: proof}.

\section{Notation and preliminaries}
\label{sec: notation}

\subsection{Notation and background on continuum random trees}
We only give here a short overview, the interested reader may consult \cite{aldcrt3}, \cite{Legall2005}
or \cite{evans05} for more details. 

\paragraph{Measured metric spaces. }
A \emph{pointed measured metric space} is a quadruple $(X, d, \mu, \rho)$ where $(X, d)$ is a compact metric space 
equipped with a \emph{finite} Borel measure $\mu$, and $\rho$ is a distinguished point that is usually referred to 
as being the root. 
Two pointed measured metric spaces $(X, d, \mu,\rho)$ and $(X', d', \mu',\rho')$ are  \emph{equivalent} 
if there exists an isometry $f: (X, d)\to (X', d')$ satisfying $\mu'=f_\ast\mu$ and $f(\rho)=\rho'$. Let 
$\mathbb M$ denote the set of equivalence classes of pointed measured metric spaces. Then 
$\mathbb M$ is a Polish space when endowed with the pointed Gromov--Hausdorff--Prokhorov topology (\cite{evans05,miermont09}). 
%\fnic{I would go for GP only: the Gromov's reconstruction relies on the fact that $\mu$ has 
%full support, and if we use GHP, then we need to discuss the space of mm space for which the measure has 
%full support, and this will dilute the focus. + if the Polish character is only used to deal with the point measures in the decompositions, while Gromov's reconstruction is not formally used, then maybe it is worth mentioning somewhere. (but I don't say that here is the best place; must be considered.)
%\comm{I'm afraid I have to argue this. Even if we use GP, we need to suppose $\mu$ to have full support; otherwise there will be a problem if the root is not in the support (similarly, each time we take points in $X$). The polish character is needed to make sense of ``random measured metric space'', but yes, we also need it to talk of point measures on $\bbM$. {\nic OK}}}

The following functional defined on $\bbM$ is useful in our treatment.  
Let $(X, d, \mu, \rho)$ be a pointed measured metric space where $\mu$ is a probability measure and has $X$ as its support. Let $(\eta_i, i\ge 1)$ be a sequence of i.i.d.~points of $X$ with common distribution $\mu$ and set $\eta_0=\rho$. We define a random symmetric and semi-infinite matrix 
\[
\mathrm M(X, d, \mu, \rho)=(d(\eta_i, \eta_j): 0\le i, j<\infty). 
\]
Note that the distribution of  $\mathrm M(X, d, \mu, \rho)$ only  depends on the equivalence class of $(X, d, \mu, \rho)$. 
Thanks to Gromov's reconstruction theorem (\cite[Section $3\frac{1}{2}.5$]{Gromov}), $\mathrm M(X, d, \mu, \rho)$ characterizes this equivalence class. In the rest of the paper and when no confusion arises, we often use the short-hand notation $\mathrm X=(X, d, \mu, \rho)$ to indicate that $\mathrm X$ stands for the whole equivalence class of $(X, d, \mu, \rho)$. 

\paragraph{Real trees. }
The metric spaces of interest here are real trees. A compact metric space $(X,d)$ is a \emph{real tree} if for any two points $u, v\in X$, the following two properties hold.  
First, there exists a unique isometry $\varphi: [0, d(u, v)]\to X$ such that $\varphi(0)=u$ and $\varphi(d(u, v))=v$; in this case, we denote by $\llb u, v\rrb_X:=\varphi([0, d(u, v)])$ or sometimes simply by $\llb u, v\rrb$ if the underlined metric space $(X, d)$ is clear from the context.
Second, if $f: [0, 1]\to X$ is an injective continuous map satisfying $f(0)=u$ and $f(1)=v$, then necessarily $f([0, 1])=\llb u, v\rrb_X$. A \emph{rooted real tree} $(T, d, \rho)$ is a real tree $(T, d)$ with a distinguished point called  the \emph{root}. 

Let $(T, d, \rho)$ be a rooted real tree. The \emph{degree} of a point $u\in T$, which we denote by $\deg(u, T)$, is the number of connected components of $T\!\setminus\!\{u\}$.  We let
\[
\Lf(T)=\{u\in T: \deg(u, T)=1\}, \quad 
\Br(T)=\{u\in T: \deg(u, T)\ge 3\} \;\text{ and }\; \Sk(T)=T\!\setminus\!\Lf(T)
\]
denote the set of the \emph{leaves},  the set of \emph{branch points} and the \emph{skeleton} of $T$, respectively. 
Note that the distance $d$ induces a sigma-finite measure $\ell$ on $T$ satisfying $\ell(\llb u, v\rrb)=d(u, v)$,  for any $u, v\in T$. We refer to $\ell$ as the \emph{length measure} of $T$. A \emph{subtree} $S$ of $T$ is a closed and connected nonempty subset of $T$. Observe that $(S, d)$ is itself a real tree. We often root $S$ at the point $\Rt(S)$, which is defined to be the unique point of $S$ minimising the distance to $\rho$\,; in that case, we say that $S$ is a \emph{rooted subtree}. 

%A \emph{measured real tree} is a pointed measured metric space $(T, d, \mu, \rho)$ where $(T, d)$ is a real tree. 
Let $(T, d, \rho)$ be a rooted real tree and let $u, v\in T$ be two distinct points.  We denote by $\SSub(T, u \,|\,v)$ the connected component of $T\!\setminus\!\{u\}$ which contains $v$.  Namely, 
\begin{equation}\label{def: SS}
\SSub(T, u \,|\,v)=\{ w\in T: u\notin \llb w, v\rrb\}. 
\end{equation}
Now let  $\bv=\{v_1, \dots, v_k\}$ be a set of $k$ points of $T$. We write 
\[
\Span(T,  \bv)=\cup_{1\le i\le k}\llb \rho, v_i\rrb
\] 
for the subtree of $T$ spanning $v_1, \dots, v_k$. %Then $\Span(T,\bv)$ is a compact metric space with the metric inherited from $T$. 
Next, observe that there is at most countably infinite collection of the connected components of $T\!\setminus\!\Span(T, \bv)$, since $(T, d)$ is compact. 
Let $\{C^\circ_i:  i\ge 1\}$ be this collection. 
For each $i\ge 1$, let $C_i$ be the closure of $C^\circ_i$ in $T$. Then one can check that  $C_i$ is a subtree and 
there exists a unique $b_i\in \Br(T)\cap\Span(T, \bv)$ such that $C_i=C^\circ_i\cup\{b_i\}$ and $b_i=\Rt(C_i)$.  Set $h_i=d(\rho, b_i)$. 
If $T$ is further equipped with a finite (Borel) measure $\mu$, then each $C_i$ is also equipped with a finite (Borel) measure which is the restriction of $\mu$ to $C_i$; by a slight abuse of notation we still denote this measure by $\mu$. 
In that case, we denote by $\mathrm C_i=(C_i, d, \mu, b_i)$ for the (equivalence class of) pointed measured metric space, $i\ge 1$; 
then the \emph{$k$-spine decomposition} of $T$ with respect to $\bv$ is the point measure on $\bbR_+\times\bbM$ defined as
\begin{equation}\label{def: k-spine}
\D(T, \bv)=\sum_{i\ge 1} \delta_{(h_i,\, \mathrm C_i)}
\end{equation}
%{\red add $(C_i,d_i,\mu,b_i)$ is an isometry class.}
%Note that $T$ can be obtained from $\Span(T, \bv)$ and $\D(T, \bv)$ by identifying the root of $C_i$ with the point $b_i\in \Span(T, \bv)$. 

A \emph{measured real tree} is a pointed measured metric space $(T, d, \mu, \rho)$ where $(T, d)$ is a real tree. For instance, each $\mathrm C_i$ in \eqref{def: k-spine} is a measured real tree.

\paragraph{The Brownian continuum random tree. }\label{p: Q}
One way to obtain a measured real tree starts from an \emph{excursion}: a continuous nonnegative  function $e: \bbR_+\to \bbR_+$ is said to be an excursion if $e$ has compact support and satisfies that $e_0=0$, $\zeta_e\!:=\sup\{s>0: e_s>0\}\in (0, \infty)$ and $e_s>0$, $\forall s\in (0, \zeta_e)$. Let $e$ be an excursion. For  $s, t\in [0, \zeta_e]$, let 
\begin{equation}\label{eq: def_d_e}
\tfrac{1}{2}\,d_e(s, t):=e_s+e_t-2\inf_{u\in [s\wedge t,\, s\vee t]}e_u.
\end{equation} 
The factor $1/2$ in the above definition is unconventional but suits our purpose here. 
Define $s\sim_e t$ if $d_e(s, t)=0$. Then $d_e$ induces a metric on the quotient space $[0, \zeta_e]/\!\!\sim_e$, which we still 
denote by $d_e$. Moreover, the metric space 
$(\cT_e:=[0, \zeta_e]/\!\!\sim, d_e)$
is a real tree (see e.g.~Theorem 2.1 
of \cite{Duquesne05}). Write $p_e: [0, \zeta_e]\to \cT$ for the canonical projection. We denote by $\mu_e$ the push-forward of the Lebesgue measure on $[0, \zeta_e]$ by $p_e$ and set $\rho_e:=p_e(0)$. Then 
$(\cT_e, d_e, \mu_e, \rho_e)$ is a measured real tree as defined previously. Moreover, it follows from the above construction that the support of $\mu_e$ is $\cT_e$ and that  $\mu_e(\cT_e)=\zeta_e$. 

Let $e$ denote the canonical process of $\mathbf C(\bbR_+, \bbR)$. 
For $a\in (0, \infty)$, let $\bP^{(a)}$ be the probability distribution on $\mathbf C(\bbR_+, \bbR)$ of the normalized Brownian excursion of  
length $a$, namely, $e$ under $\bP^{(a)}$ is distributed 
as a Brownian excursion conditioned on $\zeta_e=a$. The following \emph{scaling property} of Brownian excursions plays a crucial role in our treatment: for each $a>0$, 
\begin{equation}\label{eq: scaling}
\Big(\tfrac{1}{\sqrt{a}} e_{at}, t\ge 0\Big) \text{ under } \bP^{(a)} \; \eqd \; e \text{ under } \bP^{(1)}. 
\end{equation}
Recall the measured real tree $(\cT_e, d_e, \mu_e, \rho_e)$ from the paragraph above. 
We view (the equivalence class of) $(\cT_e, d_e, \mu_e, \rho_e)$ under $\bP^{(a)}$ as a random variable taking values in $\bbM$, whose distribution we still denote as $\bP^{(a)}$. In particular,  $\bP:=\bP^{(1)}$ is the \emph{law of the Brownian continuum random tree}. 
A real-valued random variable $\rR$ is a \emph{Rayleigh random variable} if $\rR$ has density $\mathbf 1_{\{x>0\}}xe^{-x^2/2}$. The following well-known fact will be used implicitly at various places: let $\eta\in \cT_e$ be a random point of distribution $\mu_e$ and let $\eta'$ be either another independent point of distribution $\mu_e$ or the root $\rho$; then, under $\bP$, in both cases $d_e(\eta, \eta')$ is a Rayleigh random variable.

%\paragraph{Remarks on the double randomness and some measurability issues.} 
In this work, we study stochastic processes defined on random measured real trees. In a general way, we construct these processes first for the canonical process $e$, or equivalently the real tree $(\cT_e, d_e, \mu_e, \rho_e)$; we then consider the ensemble under the law $\bP^{(a)}$, for $a>0$. See e.g.~\cite{AbSe02} for a construction of the Aldous--Pitman fragmentation in this manner. 
In the rest of the paper, %$\bcT$ stands for the equivalence class of the measured real tree 
$(\cT, d, \mu, \rho)$ stands for $(\cT_e, d_e, \mu_e, \rho_e)$ and $\bcT$ the corresponding measured real tree.

%In the rest of the paper, $\bcT=(\cT, d, \mu, \rho)$ stands for $(\cT_e, d_e, \mu_e, \rho_e)$.
%If $(X, d, \mu, \rho)$ is an equivalence class of pointed measured metric space, in what follows we often do not distinguish the difference between $(X, d, \mu, \rho)$ and a particular element of the class. We believe this will not cause confusion since the objects we consider can be statistics of the ``finite dimensional marginals'' :
%pick random points $x,y,z$ according to $\mu$, let $b$ be the branchpoint between the points $x,y,z$; 
%output the distances $d(x,y)$, $d(x,b)$ and the likes; 
%output the $\mu$-measure of the connected component of $T\{b\}$ containing $x$;
%output the isometry class of this same component, rooted and pointed at $b$ and $x$;
%and the likes.

\subsection{Cut tree of the Brownian continuum random tree}\label{sec:A-P}

Let $\bcT=(\cT, d, \mu, \rho)$ be as defined above, where $\cT$ is further equipped with the length measure $\ell$. Recall that $\bP$ is the law of the Brownian CRT. 
We define the cut tree for $\bcT$, following Bertoin and Miermont \cite{Bert12}. 
To that end, let $\cP$ be a Poisson point process on $\bbR_+\times \cT$ of intensity measure $dt\otimes \ell(dx)$. Every point $(t,x)\in \cP$ is seen as a \emph{cut} on $\cT$ at location~$x$ and  
arriving at time~$t$. Given $(\cT,d,\mu,\rho)$, let $(V_i)_{i\ge 1}$ be a sequence of i.i.d. points of $\cT$ with common distribution $\mu(\cdot)/\mu(\cT)$. Then for each $i\ge 1$ and $t\ge 0$, let $\cT_i(t)$ be the set of those points in $\cT$ which are still connected to $V_i$ at time $t$, that is
\begin{equation}\label{def: Tt}
\cT_i(t):=\{u\in \cT: \cP\cap ([0, t]\times \llb V_i, u\rrb)=\varnothing\}.
\end{equation}
For $i\ne j$, set $t_{ij}=\inf\{t\ge 0: \cP\cap ([0, t]\times \llb V_i, V_j\rrb)\ne \varnothing\}\in [0, \infty]$. 
Define a symmetric function $\delta: \mathbb Z_+^2 \to \bbR_+$ by setting $\delta(i, i)=0$ for $i\ge 0$ and 
\begin{equation}
\label{def: delta}
\delta(0, i) = \int_0^\infty \mu\big(\cT_i(s)\big)ds, \quad 
\delta(i, j) = \int_{t_{ij}}^\infty \Big\{\mu\big(\cT_i(s)\big)+\mu\big(\cT_j(s)\big)\Big\} ds, \quad i, j\ge 1, i\ne j. 
\end{equation}
\begin{prop}[\cite{Bert12}]\label{prop: cG}
Under $\bP$, the following statements (I-II) hold almost surely.
\begin{compactenum}[I.]
\item
For all $i, j\ge 1$ and $i\ne j$, $\delta(0, i)\in (0, \infty)$,  $t_{ij}\in (0, \infty)$, and $\delta(i, j)\in (0, \infty)$. 
\item
There exists a measured real tree $\cut(\cT)=(\cG, d_{\cG}, \mu_{\cG}, \rho_{\cG})$ and a sequence of its points $(V'_i)_{i\ge 0}$ with $V'_0=\rho_{\cG}$ such that  
\[
d_{\cG}(V'_i, V'_j)\, = \, \delta(i, j), \quad \forall\, 0\le i, j<\infty.
\]
Also, conditional on $\mu_{\cG}$\,, $(V'_i)_{i\ge 1}$ has the distribution of a sequence of i.i.d.~points with common probability distribution $\mu_{\cG}$. 
\end{compactenum}
Moreover, $\cut(\cT)$ under $\bP$ has the same distribution as $\bcT$ under $\bP$. 
\end{prop}

Part of the above Proposition says that $(\cG, d_{\cG}, \mu_{\cG}, \rho_{\cG})$ is uniquely determined by $(\delta(i, j): 0\le i, j<\infty)$ (up to measure-preserving isometry), by Gromov's reconstruction theorem. Therefore, the measured real tree $\cut(\cT)$ is well-defined, $\bP$-a.s.
It also follows from the above construction that the mapping $(\cT, \cP)\mapsto \cut(\cT)$ is measurable; see a related discussion in \cite{Bert12}. 

If, for $t\ge 0$, we write $i\sim_t j$ if and only if $\cT_i(t)=\cT_j(t)$, then $\sim_t$ defines an exchangeable  random partition of $\bbN$. Moreover, the family of partitions $\{\sim_t: t\ge 0\}$ has a natural genealogical structure, which is described by $\cut(\cT)$; we refer to \cite{Bert12, bertoinfrag} for more details.  
Note that the root of $\cT$ is irrelevant in the above cutting process, whereas the root of $\cut(\cT)$ is meaningful for the genealogy it describes.

\subsection{Partial cut trees as an approximation of $\cut(\cT)$}

We recall here the definition of the $k$-partial cut tree of the measured real tree $\bcT=(\cT, d, \mu, \rho)$, as well as some of its properties, which will be useful for the proof later. These properties are mostly proven in \cite{ABH10} and \cite{BW} for $k=1$, but the case $k>1$ in general also follows from the arguments there. 

For $i, j\ge 1$ and $t\ge 0$, recall $\cT_i(t)$ from \eqref{def: Tt} and $t_{i, j}$ just below. Observe that  $\cT_i(t)\subseteq \cT_i(t')$ if $t\ge t'$. We then denote $\cT_i(t-)=\cap_{s<t} \cT_i(s)$, for $t>0$. 
Set $t^\ast_i=\max_{1\le j<i} t_{i, j}$ for $i\ge 2$ and $t^\ast_1=0$. 
For each $i\ge 1$, 
%let $\cJ_i$ be a countable set such that $\cJ_i, i\ge 1$ are all disjoint and that 
let $\{t^{_{(i)}}_m: m\in \N\}$ be the set of discontinuity points of the mapping $t\in (t^\ast_i, \infty)\mapsto \mu(\cT_i(t))$. Then for each $m\in \N$, almost surely  there exists a unique point $x^{_{(i)}}_m\in \cT_i(t^{_{(i)}}_m-)$ 
such that  $(t^{_{(i)}}_m, x^{_{(i)}}_m)\in \cP$. 
It follows that $\Delta^{_{\circ, (i)}}_m:=\cT_{i}(t^{_{(i)}}_m-)\setminus\cT_i(t^{_{(i)}}_m)$ is non empty and $x^{_{(i)}}_m\in \Delta^{_{\circ,(i)}}_m$. Let $\Delta^{_{(i)}}_m$ be the closure of $\Delta^{_{\circ, (i)}}_m$ and let $\mathbf\Delta^{_{(i)}}_m=(\Delta^{_{(i)}}_m, d, \mu, x^{_{(i)}}_m)$ be the (equivalence class of) measured real tree induced. Note in particular that $\mu(\Delta^{_{(i)}}_m)=\mu(\Delta^{_{\circ,(i)}}_m)$. 
% the measured real tree obtained from the closure of $(\Delta^{_{\circ, (i)}}_m, d, \mu)$. %Observe that $x^{_{(i)}}_m\in \Delta^{_{(i)}}_m$. 
We also define $h^{_{(i)}}_m:=\int_{[0, \, t^{_{(i)}}_m]} \mu(\cT_i(s))ds$ for each $m\in \N, i\ge 1$. 
Let us recall from \eqref{def: k-spine} the $k$-spine decomposition of a real tree.  
\begin{prop}[\cite{BW}]\label{prop: cG_k}
Under $\bP$, the following holds almost surely: 
for each $k\ge 1$, there exists a measured real tree $\cut(\cT; V_1, \dots, V_k)=(\cG_k, d_{\cG_k}, \mu_{\cG_k}, \rho_{\cG_k})$ and $k$ points $V'_1, \dots, V'_k\in \cG_k$ such that 
\begin{align*}
&\Span(\cG_k, \{V'_1, \dots, V'_k\})  \text{ is isometric to }\Span(\cG, \{ V'_1, \dots, V'_k\})\,,    \\
&\D(\cG_k, \{V'_1, \dots, V'_k\})=\sum_{1\le i\le k}\sum_{m\in \N} \delta_{(h^{_{(i)}}_m, \,\mathbf\Delta^{_{(i)}}_m)}\, , 
\end{align*}
where the real tree $(\cG, d_{\cG})$ is defined in Proposition \ref{prop: cG}; conditional on $\mu_{\cG_k}$, $(V'_1, \dots, V'_k)$ are distributed as $k$ independent points of common probability distribution $\mu_{\cG_k}$. 
Moreover for each $k$, $\cut(\cT; V_1, \dots, V_k)$ under $\bP$ has the same distribution as $\bcT$ under $\bP$. %, and conditional on $\cut(\cT; V_1, \dots, V_k)$, $(V'_1, \dots, V'_k)$ are distributed as $k$ independent points of common probability distribution $\mu_{\cG_k}$. 
\end{prop}

The case $k=1$ of Proposition \ref{prop: cG_k} corresponds to Theorem 1.7 of \cite{ABH10} (see also Theorem 3.2 of \cite{BW}). The arguments there can be straightforwardly adapted to yield a proof of the general case $k\ge 1$. Now recall from page \pageref{p: Q} the probability measure $\bP^{(a)}$ for the measured real tree encoded by a Brownian excursion of length $a$. Recall also \eqref{eq: scaling}, the scaling property of Brownian excursions. The following is a direct consequence of Proposition \ref{prop: cG_k} and a multi-point version of the Bismut decomposition for the Brownian CRT (\cite[Theorem 3]{Legall1993}). 

\begin{cor}[Scaling property]\label{cor: 1}
Let $k\ge 1$. For $1\le i\le k$ and $m\in \N$, denote by $\mu_{i, m}=\mu(\Delta^{_{(i)}}_m)$. Then under $\bP$, conditional on the collection $\{\mu_{i, m}: 1\le i\le k, m\in \N\}$, the measured real trees $\{\mathbf\Delta^{_{(i)}}_m:  1\le i\le k, m\in \N\}$ are independent and $\mathbf\Delta^{_{(i)}}_m$ has distribution $\bP^{(\mu_{i, m})}$. 
\end{cor}

The following is the analog of Algorithm \ref{alg: k-cut} for the continuum trees. 
\begin{lem}[Recurrence relation for $(\cG_k)$]\label{lem: rec}
Let $k\ge 2$. Let $i_{k}=\min\{i\in \{1, \dots, k-1\}: \exists \, m\in \N \text{ such that } V_{k}\in \Delta^{_{(i)}}_m\}$ and let $m_{k}\in \N$ be the index such that $V_{k}\in \Delta^{_{(i_k)}}_{m_k}$. Then under $\bP$, we have a.s.
\[
\D(\cG_{k}, \{V'_1, \dots, V'_{k-1}\})=\sum_{\substack{1\le i\le k:\\ i\ne i_k}}\sum_{m\in \N} \delta_{\big(h^{_{(i)}}_m, \,\mathbf\Delta^{_{(i)}}_m\big)}+\sum_{\substack{m\in \N: \\m\ne m_k}}\delta_{\big(h^{{(i_k)}}_m,\, \mathbf\Delta^{{(i_k)}}_m\big)}+\delta_{\big(h^{{(i_{k})}}_{{m_k}}, \,\cut(\mathbf\Delta^{{(i_k)}}_{m_k}; V_{k})\big)}.
\]
\end{lem}

Note that in the above formula, $\cut(\Delta^{{(i_k)}}_{m_k}; V_{k})$ is well-defined under $\bP$ thanks to Corollary~\ref{cor: 1}. Lemma~\ref{lem: rec} is easily seen to hold true by comparing the definitions of $\cG_k$ and $\cG_{k-1}$. 

The following observation constitutes the foundations of our partial reconstructions. 
\begin{lem}\label{lem: key}
Let $k\ge 1$ and let $\eta, \eta'$ be two independent points of $\cT$ sampled according to the distribution $\mu(\cdot)/\mu(\cT)$. We have the following. 
\begin{enumerate}[a)]
\item
Set $\cI_k(\eta, \eta')=\{(i, m): 1\le i\le k, m\in \N, \Delta^{_{\circ,(i)}}_m\cap\llb \eta, \eta'\rrb_{\cT}\ne \varnothing\}$.  Then $\bE( |\cI_k(\eta, \eta')|) <\infty$. 
\item
For $(i, m)\in \cI_k(\eta, \eta')$,  set 
$\rR'_{i, m}=\mu(\Delta^{_{(i)}}_m)^{-\frac{1}{2}}\cdot\ell(\Delta^{_{\circ,(i)}}_m\cap\llb \eta, \eta'\rrb_{\cT})$. Under $\bP$, conditional on the set $\cI_k(\eta, \eta')$, $\{\rR'_{i, m}: (i, m)\in \cI_k(\eta, \eta')\}$ are independent Rayleigh random variables which are independent of the collection $\{\mu(\Delta^{_{(i)}}_m): 1\le i\le k, m\in \N\}$. 
\end{enumerate}
\end{lem}

\begin{proof}
Proof of a). For $i\ge 1$, let $b_i$ be the unique  point of $\cT$  satisfying $\llb b_i, V_i\rrb=\llb \eta, V_i\rrb\cap \llb \eta', V_i\rrb$, and let $\tau_i=\inf\{t>0: \cP\cap ([0, t] \times \llb b, V_i\rrb)\ne\varnothing\}$, the time of the first cut on $\llb b_i, V_i\rrb$. Observe that $\tau_i$ is an exponential random variable of parameter $d(V_i, b_i)$. Let $\tau=\max_{1\le i\le k}\tau_i$, which has finite expectation under $\bP$.
Denote by $N$ the cardinality of $\cP\cap([0, \tau]\cap \llb \eta, \eta'\rrb)$; then $N$ is distributed as a Poisson random variable of rate $\tau\cdot d(\eta, \eta')$. Note that $|\cI_k(\eta, \eta')|$ is stochastically dominated by $N$, since for all $t>\tau$, $\llb \eta, \eta'\rrb \cap \cT_i(t)=\varnothing$ for all $i=1, \dots, k$. This yields $\bE(|\cI_k(\eta, \eta')|)<\infty$. 

Proof of b). The case $k=1$ is a consequence of Theorem 5.1 in \cite{ABH10}. The general case follows by adapting the arguments there and we omit the formal proof.
\end{proof}

Lemma \ref{lem: key} says that $\llb \eta, \eta'\rrb$ only intersects a finite sub-collection of $\{\Delta^{_{ (i)}}_m: 1\le i\le k, m\in \N\}$. This suggests that to reverse the mapping of the $k$-partial cut tree $\bcT\mapsto \cut(\cT; V_1, \dots, V_k)$, which boils down to ``reconstructing'' the distance $d(\eta, \eta')$ from the partial cut tree, we should first sample the collection $\{\Delta^{(i)}_m: (i, m)\in \cI_k(\eta, \eta')\}$ from the $k$-spine decomposition of $\cG_k$. 
In the next section, we develop this idea into a definition of the partial shuffle transforms.

\section{The shuffle transform}
\label{sec: shuff}

In this section, we give the definition of the shuffle transform by generalizing the construction in Section~\ref{sec: intro_discret}. Recall that $\bcT=(\cT, d, \mu, \rho)$ is a measured real tree and $\bP$ is the law of the Brownian CRT. We aim at defining for each $k\ge 1$, a (random) semi-infinite matrix $\Gamma_k=(\gamma_k(i, j): 1\le i, j<\infty)$ which will play the role of the $k$-partial reversal for $\bcT$. Indeed, 
$\gamma_k(i, j)$ will represent the distance between two independent uniform points $\eta_i, \eta_j$ obtained from the $k$-partial reversal. The main theorem (Theorem \ref{thm: cv_gamma}) then states that under $\bP$,  the sequence $(\gamma_k(i, j))_{k\ge 1}$ converges almost surely to a limit $\gamma_\infty(i, j)$, for all $(i, j)$. Moreover, the limit $(\gamma_\infty(i, j): 1\le i, j<\infty)$ characterizes a measured real tree which will be the image of $\cT$ by the shuffle transform.  

To define $\Gamma_k$, we do the following. First, for each $k$, we sample a collection of random points or marks, in an analogous way as we have done for Cayley trees. We then explain how to build a path between two independent points in the $1$-partial reversal using these marks. 
See Fig.~\ref{fig:path}. 
Relying on a recursive procedure, this construction is then extended to more general partial reversals, which gives us the definition of $\Gamma_k$. %In what follows, it is often convenient to suppose that the elements of a set $A$ are indexed by the set $\cJ(A)$, which are supposed to be disjoint for different $A$.

\paragraph{Sampling the marks. }
Let $(V_i)_{i\ge 1}$ be a sequence of i.i.d.\ points of $\cT$ whose common distribution is $\mu(\cdot)/\mu(\cT)$. For $k\ge 1$, write $\bV_k=\{V_1, \dots, V_k\}$ and then $\cS_k=\Span(\cT, \bV_k)$. Recall from \eqref{def: SS} the notation $\SSub(\cT, u \,|\, v)$. Set $\cT_1=\cT$. For $k\ge 2$, with probability $1$, there is a unique connected component of $\cT\setminus\cS_{k-1}$ containing $V_{k}$; set $\cT_k$ to be the smallest rooted subtree of $\cT$ containing that connected component. 
% and  let $B_{k+1}$ be the unique element of $\cT_{k+1}\cap\cS_k$. 
%We regard each $(\cT_k, d, \mu, B_k)$ as a measured real tree. 
We define the sequences $(\rho_k)_{k\ge 1}$ and $(\cM_k)_{k\ge 1}$ as follows:
\begin{itemize}
\item
For each $k\ge 1$, let $\rho_k\in \cT_k$ be a random point having the following distribution
\begin{equation}\label{def: rho}
\forall \text{ Borel set } B\subset \cT : \; \bP(\rho_k\in B)=\mu|_{\cT_k}(B), \quad \text{ where } \mu|_{A}:=\frac{\mu(\cdot\cap A)}{\mu(A)}, \; A\subset \cT. 
\end{equation}
\item
For each $k\ge 1$, let  $\{C_{k, i}^\circ: i\in \N\}$ be the collection of the connected components of $\cT_k\setminus\Span(\cT_k, \{V_k\})$ and let $C_{k, i}$ be the smallest rooted subtree of $\cT$ containing $C_{k, i}^\circ$. Note that $b_{k, i}:=\Rt(C_{k, i})$ is the only element of $\cS_k\cap C_{k, i}$. 
For each $i\ge 1$, let $U_{k, i}\in \SSub(\cT_k, b_{k, i} \,|\, V_k)$ be a random point with distribution $\mu|_{\SSub(\cT_k,\, b_{k, i} \,|\,V_k)}$.  
Observe that $\{C_{k, i}: i\ge 1\}$ is a collection of disjoint rooted subtrees of $\cT_k$ and almost surely $U_i\in C_{k, i'}$ for some $i'\ne i$. 
We define $\cM_k$ to be the collection 
\begin{equation}\label{def: cM_k}
\cM_k=\{(C_{k, i}, U_{k, i}): i\in \N\}. 
\end{equation}
%where $\{C_i: i\ge 1\}$ is a collection of disjoint subtrees of $\cT$ and $U_i\in C_{i'}$ for some $i'\ne i$, almost surely. 
\end{itemize}

\paragraph{Building paths from the marks. } 
Suppose that we have a collection %that $v\in \cT$ and that 
\[
\cN=\{(C_i,\,u_i): i\ge 1\}, 
\]
where $\{C_i: i\ge 1\}$ is a collection of disjoint rooted subtrees of $\cT$ and $u_i\in C_{i'}$ for some $i'\ne i$, for all $i\ge 1$. 
Then for each $u\in \cup_{i\ge 1} C_i$, we introduce the following sequence
%point measure on $\bbM\times \cT$ such that $\sum_{i\in \cI(\cT, v)} \delta_{(d(\rho,\, b_i), \,(C_i, \,d, \,\mu, \,b_i))}$ is the $1$-spine decomposition of $\cT$ with respect to $v$ and such that $u_i\in \SSub(\cT, b_i \,|\, v)$, for all $i\in \cI(\cT, v)$. 
%Then for each $u\in \cT\!\setminus\!\llb \rho, v\rrb$, we introduce the following sequence
\[
\chi^u(\cT, \cN)=\big\{(C^u_j, a^u_j, p^u_j):  j\ge 1\big\}, 
\]
which is defined in the following inductive way: 
$a^u_1=u$ and for each $j\ge 1$, 
\begin{equation}\label{eq: mk_1}
\text{let } i_j \text{ be the index such that } a^u_j\in C_{i_j}, \text{ then }\; 
C^u_j=C_{i_j}, \; p^u_j=\Rt(C_{i_j}) \,\text{ and }\, a^u_{j+1}=u_{i_j}\,.
\end{equation}
Note that such an $i_j$ exists since $u$ and all $u_i$ belong to $\cup_i C_i$ and $i_j$ is unique as the $C_i$'s are disjoint.  
Next, let $u, u'\in\cup_i C_i$ be two distinct points; set 
\begin{equation}\label{def: tau}
\sI(u, u'; \cT, \cN)=\inf\big\{j\ge 1: C^u_j\in \{C^{u'}_{j'}:  j' \ge 1\}\big\},
\end{equation}
with the convention that $\inf\varnothing=\infty$. Observe that $\sI(u, u'; \cT, \cN)<\infty$ if and only if $\sI(u', u; \cT, \cN)<\infty$; in that case,  $C^u_{\sI(u, u'; \cT, \cN)}=C^{u'}_{\sI(u', u; \cT, \cN)}$. 
Let $\tilde\chi(u, u'; \cT, \cN)$ be a (possibly infinite) collection
\begin{equation}\label{def: chi}
\tilde\chi(u, u'; \cT, \cN)=\big\{\big(C^{u, u'}_m, a^{u, u'}_m, p^{u, u'}_m\big): m\in \N, m\le N^{u, u'}\big\},
\end{equation}
where $N^{u, u'}=|\tilde\chi(u, u'; \cT, \cN)|\in \N\cup\{\infty\}$ and $\{C^{u, u'}_m: m\in \N, m\le N^{u, u'}\}$ consists of disjoint rooted subtrees. We further requires $\tilde\chi(u, u'; \cT, \cN)$ to satisfy the following conditions \eqref{cd1} and \eqref{cd2}: 
\begin{align}\label{cd1}
&\big\{C^{u, u'}_m: m\in \N, m\le N^{u, u'}\big\}\\ \notag
&\qquad\quad=\big\{C^u_j: 1\le j< \sI(u, u'; \cT, \cN)\}\cup\{C^{u'}_j: 1\le j<\sI(u', u; \cT, \cN)\big\}\cup\{C^u_{\sI(u, u'; \cT, \cN)}\},
\end{align}
where the last term is taken to be $\varnothing$ if $\sI(u, u'; \cT, \cN)=\infty$.  
And 
\begin{align}\notag
%\left\{
&\text{{\bf case 1:}
if $C^{u, u'}_m=C^u_j$ with $1\le j<\sI(u, u'; \cT, \cN)$,  then $a^{u, u'}_m=a^u_j$ and $p^{u, u'}_m=p^u_j$};\\ \label{cd2}
&\text{{\bf case 2:}
if $C^{u, u'}_m=C^{u'}_j$ with $1\le j<\sI(u', u; \cT, \cN)$,  then $a^{u, u'}_m=a^{u'}_j$ and $p^{u, u'}_m=p^{u'}_j$};\\ \notag
&\text{{\bf case 3:}
if $\sI(u, u'; \cT, \cN)<\infty$ and $C^{u, u'}_m=C^u_{\sI(u, u'; \cT, \,\cN)}$, then $a^{u, u'}_m=a^u_{\sI(u, u'; \cT,\, \cN)}$}\\ \notag
&\qquad\quad \text{ and $p^{u, u'}_m=a^{u'}_{\sI(u', u; \cT,\, \cN)}$}.
%\right.
\end{align}
For the sake of definiteness, we can require the elements of $\tilde\chi(u, u'; \cT, \cN)$ to be ranked in decreasing order according to $\mu(C^{u, u'}_m)$, say, but the order of $\tilde\chi(u, u'; \cT, \cN)$ is irrelevant for the rest of the construction.  
%Observe that if $\sI(u, u'; \cT, \cN)<\infty$, then also $\sI(u', u; \cT, \cN)<\infty$ and therefore $\tilde\chi(u, u'; \cT, \cN)$ is a finite set in this case. 
The above definition can be seen as an analog of the 1-partial shuffle transform for Cayley trees, which has been illustrated in Figure \ref{fig:path}. Indeed, each $C_i$ in the collection $\mathcal N$ can be understood as a subtree isolated from the rest of the tree due to the cutting at its root and the point $u_i$ represents the neighbor of this cut in the original tree. Next, we choose a subcollection of subtrees  $(C^{u, u'}_m)_{m\ge 1}$ which contain a portion of the path between $u$ to $u'$ in the original tree (i.e.~the white trees in Figure \ref{fig:path}). The subcollection is chosen by following this path: we start from $u=a^u_1$ and look for the nearest cut to $u$ (i.e.~$p^u_1$) on the path; the subtree rooted at $p^u_1$ is then $C^u_1$ and the neighbor of $p^u_1$ is $a^u_2$, etc. Proceeding in the same way from $u'$ yields another sequence $\chi^{u'}(\cT, \cN)$. Merging the two sequences up to the point where they coincide gives  $\tilde\chi(u, u'; \cT, \cN)$.

\paragraph{Defining partial reversals. }
Let $(\rho_k)_{k\ge 1}$ and $(\cM_k)_{k\ge 1}$ be as  defined in \eqref{def: rho} and \eqref{def: cM_k}.  
Let $(\eta_i)_{i\ge 2}$ be an independent sequence of i.i.d.~points of common distribution $\mu|_{\cT}$ and set $\eta_1=\rho_1$. 
Recall the definition of  $\tilde\chi(u, u'; \cT, \cN)$ from \eqref{def: chi}. Recall $\cT_k$ is a rooted subtree containing $V_k$. 
For all $i, i'\ge 1$ such that $i\ne i'$, we define
\begin{equation}\label{def: chi'}
\chi_k^{i, i'}=\big\{(C^{i, i'}_{k, m}, a^{i, i'}_{k, m}, p^{i, i'}_{k, m}): m\in\N, m\le N^{i, i'}_k\big\}, \quad k\ge 1,
\end{equation}
in the following inductive way. Let $\chi_1^{_{i, i'}}=\tilde\chi(\eta_i, \eta_{i'}; \cT, \cM_1)$, which is well-defined since almost surely we have $\eta_i, \eta_{i'}\in \cT\!\setminus\!\cS_1\subset \cup_i C_{k, i}$.
Suppose that $\chi_{k}^{_{i, i'}}$ has been defined. If $\cT_{k+1}\notin \{C^{_{i, i'}}_{k,m}:  m\in \N,  m\le N^{_{i, i'}}_k\}$, then we set $\chi_{k+1}^{_{i, i'}}=\chi_{k}^{_{i, i'}}$. Otherwise, let $m_k$ be the index such that $\cT_{k+1}
=C^{_{i, i'}}_{k, m_k}$; then we define $\chi^{_{i, i'}}_{k+1}$ to be the collection 
\begin{equation}\label{def: chi''}
\chi^{i, i'}_{k+1}=\big\{(C^{i, i'}_{k, m}, a^{i, i'}_{k, m}, p^{i, i'}_{k, m}): m\in \N\!\setminus\!\{m_k\}, m\le N^{i, i'}_k\big\}\cup\tilde\chi(a^{i, i'}_{m_k}, \rho_{k+1}; \cT_{k+1}, \cM_{k+1}). 
\end{equation}
Note that $\tilde\chi(a^{i, i'}_{m_k}, \rho_{k+1}; \cT_{k+1}, \cM_{k+1})$ is well-defined, since $\cT_{k+1}$ is a rescaled version of $\cT$ and almost surely $a^{i, i'}_{m_k}, \rho_{k+1}\in \cT_k\setminus\cS_k$. 
Moreover, the role of $\rho_{k+1}$ could be understood as follows: analogously to the discrete construction (Algorithm \ref{alg: k-rev}), the ``replacement'' of $\cT_{k+1}$ will be rooted at $\rho_{k+1}$.  

For each $k\ge 1$, let us define a symmetric matrix $\Gamma_k=(\gamma_k(i, i'):  1\le i, i'<\infty\big)$ where 
\begin{equation}\label{def: gamma}
\gamma_k(i, i)=0 \quad \text{ and } \quad \gamma_k(i, i')=\sum_{m\in \N: m\le N^{i, i'}_k}d(a^{i, i'}_{k,m}, p^{i, i'}_{k,m}), \quad i\ne i'.
\end{equation}

\begin{lem}\label{lem: gamma_k}
Let $k\ge 1$, $i, i'\ge 1$ and $\gamma_k(i, i')$ be defined as in \eqref{def: gamma}. 
Then under $\bP$, we have $\gamma_k(i, i')< \infty$ almost surely. 
\end{lem}

\begin{thm}\label{thm: cv_gamma}
Under $\bP$, the following statements hold almost surely.
\begin{enumerate}[a)]
\item
The sequence of matrices $\{\Gamma_k=(\gamma_k(i, j))_{1\le i, j<\infty}: k\ge 1\}$ converges almost surely in the product topology of \,$\bbR^{\mathbb Z_+\times \mathbb Z_+}$. Denote by $\Gamma_\infty=(\gamma_\infty(i, j))_{1\le i, j<\infty}$ the almost sure limit. 
\item
There exists a measured real tree $\shuff(\cT)=(\cH, d_{\cH}, \mu_{\cH}, \rho_{\cH})$ and a sequence of its points $(\varsigma_i)_{i\ge 1}$ with $\varsigma_1=\rho_{\cH}$ such that 
\[
d_{\cH}(\varsigma_i, \varsigma_j)\, =\, \gamma_\infty(i, j), \quad \forall \, 1\le i, j<\infty.
\]
Moreover, conditional on $\cH$, $(\varsigma_i)_{i\ge 2}$ is distributed as a sequence of i.i.d.~points with common probability distribution $\mu_{\cH}$. 
\item
Set $\eta_0=\rho$. Recall from \eqref{def: delta} the matrix $(\delta(i, j))_{0\le i, j<\infty}$. 
We have 
\begin{equation}\label{id: matrix}
\big(d(\eta_i, \eta_j), \gamma_\infty(i+1, j+1)\big)_{ 0\le i, j<\infty}\text{ under } \bP \;\eqd \;\big(\delta(i, j), d(V_{i+1}, V_{j+1})\big)_{0\le i, j<\infty}  \text{ under } \bP. 
\end{equation}
In particular, this implies that the pair $(\bcT, \shuff(\cT))$ under $\bP$ has the same distribution as $(\cut(\cT), \bcT)$ under $\bP$. 
\end{enumerate}
\end{thm}

\paragraph{The mapping $\bcT\mapsto \shuff(\cT)$ is measurable. } Indeed, in the above construction, we have performed a sequence of measurable operations on $\cT$ with respect to the Gromov--Prokhorov topology. These operations can be seen as compositions of the following basic ones: 
\begin{compactenum}[--]
\item
sample two independent random points $V$ and $V'$ according to the mass measure $\mu$ and output the distances $d(\rho, V), d(V, V')$; %and denote by $b$ be the branch point between $V$ and $V'$ (i.e.~the unique element of $\llb \rho, V\rrb\cap\llb \rho, V'\rrb\cap\llb V, V'\rrb$);
\item
denote by $b$ the branch point of $V$ and $V'$, that is, the element of $\llb V, V'\rrb$ minimising the distance to the root $\rho$; 
for an independent point $\xi$ of law $\mu|_{\cT}$, determine if $\xi\in \SSub(\cT, b\, |\, V')$, which is the same to see if $d(\xi, V)+d(\rho, V')=d(\xi, \rho)+d(V, V')$; 
\item
output $\mu(\SSub(\cT, b\, |\, V'))$, which a.s.~equals $\lim_{k\to\infty} \frac{1}{k}\sum_{1\le i\le k}\mathbf{1}_{\{\xi_i\in  \SSub(\cT, b\, |\, V')\}}$, where $(\xi_i)_{i\ge 1}$ are independent points of common law $\mu|_{\cT}$; 
\item
determine if two rooted subtrees $C$ and $C'$ are identical, which reduces to compare the two matrices $\mathrm M(C, d, \mu|_{C}, \Rt(C))$ and $\mathrm M(C', d,  \mu|_{C'}, \Rt(C'))$. 
\end{compactenum}

\paragraph{Remark. }
As suggested by the discrete construction, the random points $(U_{k, i})_{k, i\ge 1}$ and $(\rho_k)_{k\ge 2}$ are the ``traces'' of the cuts left in $\cut(\cT)$. After completing an earlier version of this work, we have learned of the approach of Addario-Berry, Dieuleveut \& Goldschmidt \cite{ADG16+}, who have made a rigorous statement out of this intuition. In their work, they enrich $\cT$ with a collection of points formed by the cuts and a sequence of i.i.d.~leaves $(\xi_i)_{i\ge 1}$; similarly, $\cut(\cT)$ is enriched with the images of the cuts and $(\xi_i)_{i\ge 1}$ by the cut tree transform; then they give a reconstruction procedure (different from ours) which allows them to reconstruct almost surely the enriched $\cT$ from the enriched $\cut(\cT)$. 
Let us also remark that we can sample the randomness of the marks prior to the choice of $(V_i)_{i\ge 1}$. For this, simply sample for each branch point $b$ of $\cT$ a pair of independent points $(r^b_0, r^b_1)$, each uniformly distributed in one of the two subtrees above $b$. 
After taking $(V_i)_{i\ge 1}$, we can define the sequence $(\rho_k)$ and $(U_{k, i})_{k, i\ge 1}$ from $\{r^b_i:  b\in \Br(\cT), i=0, 1\}$ as a function of $(V_i)$. For example, $\rho_k=r^b_i$ where $b$ is the root of $\cT_k$ and $r^b_i$ is the random point associated to $b$ which belongs to $\cT_k$ (since $\cT_k$ is a subtree above $b$, there must exist one). The choice of $(U_{k, i})_{k, i\ge 1}$ is more tedious to put down; we omit it. 
This construction bears some resemblance  to the one given in \cite{ADG16+}. 

\section{Partial reversals and their convergence}
\label{sec: proof}
In this section, we give the proofs of Lemma \ref{lem: gamma_k} and Theorem \ref{thm: cv_gamma}. 
 
\subsection{Preliminary and proof of Lemma \ref{lem: gamma_k} }

Recall from \eqref{def: tau} the notation $\sI(u, u'; \cT, \cN)$. We first recall the following result, obtained  in \cite{BW}. 
\begin{lem}[\cite{BW}, Theorem 6.1]
\label{lem: tau}
If $\eta, \eta'$ are two independent points of $\cT$ with common distribution $\mu|_{\cT}$, and if $\cM_1$ is the collection in \eqref{def: cM_k}, then $\bE[\sI(\eta, \eta'; \cT, \cM_1)+\sI(\eta', \eta\,; \cT, \cM_1)]<\infty$.
\end{lem}

\begin{lem}\label{lem: chi}
Let $k\ge 1$ and $i, i'\ge 1$, $i\ne i'$. Let $\chi^{_{i, i'}}_k$  be as defined in \eqref{def: chi'}.  For each $m\in \N, m\le N^{i, i'}_k$, denote $\nu_{k, m}=\mu(C^{_{i, i'}}_{k, m})$. Then $\bE[N^{i, i'}_k]<\infty$. Moreover under $\bP$, conditional on the collection $\{\nu_{k, m}: 1\le m\le N^{i, i'}_k\}$, the collection $\chi^{_{i, i'}}_k$ consists of independent elements which are distributed as follows: 
\begin{compactenum}[(a)]
\item
$C^{_{i, i'}}_{k, m}$ has the law $\bP^{(\nu_{k, m})}$; and given $C^{_{i, i'}}_{k, m}$: 
\item
$a^{_{i, i'}}_{k, m}$ is an independent point of law $\mu|_{C^{_{i, i'}}_{k, m}}$, 
\item
$p^{_{i, i'}}_{k, m}$ is either the root of $C^{_{i, i'}}_{k, m}$ or another independent point of law $\mu|_{C^{_{i, i'}}_{k, m}}$. 
\end{compactenum}
\end{lem}

\begin{proof}
We proceed by induction on $k$. For $k=1$, by definition, $\chi^{_{i, i'}}_1=\tilde\chi(\eta_i, \eta_{i'}; \cT, \cM_1)$; then by the definition of the latter,  $N^{i, i'}_1=\sI(\eta_i, \eta_{i'}; \cT, \cM_1)+\sI(\eta_{i'}, \eta_i\,; \cT, \cM_1)-1$. It follows from Lemma \ref{lem: tau} that $\bE(N^{i, i'}_1)<\infty$. Recall from \eqref{def: k-spine} the $1$-spine decomposition of $\cT$ with respect to $V_1$.  By construction, $\{C^{_{i, i'}}_{k, m}: 1\le m\le N^{i, i'}_1\}$ is a sub-collection of $\{C_{1, i}: i\ge 1\}$, the closures of the connected components of $\cT\!\setminus\!\llb\rho, V_1\rrb$. Moreover, from the definition of $(U_{1, i})$ and the definition in \eqref{eq: mk_1} we see that the event that $C_{1, i}$ belongs to this sub-collection only depends on its $\mu$-mass. We then deduce from Corollary \ref{cor: 1} that conditional on their masses, $C^{_{i, i'}}_{k, m},  1\le m\le N^{i, i'}_1$, are independent, each one being a rescaled Brownian CRT. We next check that each $a^{_{i, i'}}_{k,m}$ is a $\mu$-random point restricted to $C^{_{i, i'}}_{k,m}$. But this  follows from  the definition of $(U_{1, i})$, \eqref{eq: mk_1} and the definition of $a^{_{u, u'}}_{k,m}$ in \eqref{cd2}. On the other hand, $p^{_{i, i'}}_{k, m}$ is either the root of $C^{_{i, i'}}_{k, m}$ ({\bf case 1} \& {\bf 2} in \eqref{cd2}) or another point independent of $a^{_{i, i'}}_{k,m}$ with distribution $\mu|_{^{C^{_{i, i'}}_{k,m}}}$ ({\bf case 3} in \eqref{cd2}). 
In this way, we verify the statements of the lemma for $k=1$. 

Now we assume that the lemma holds up to $k-1$, for some $k\ge 2$. Let us show that it also holds for $k$. Recall that $\cT_k$ is the closure of the connected component of $\cT\setminus\cS_{k-1}$ which contains $V_k$. If $\cT_k\in \{C^{_{i, i'}}_{k,m}: 1\le m\le N^{i, i'}_k \}$, then the statement for $k$ follows trivially from the induction hypothesis. If instead, $\cT_k=C^{_{i, i'}}_{k, m_{k-1}}$ with $1\le m_{k-1}\le N^{i, i'}_{k-1}$, then by the inductive hypothesis, $\cT_k$ is a rescaled Brownian CRT with total mass $\mu(\cT_k)$ and $a_k:=a^{_{i, i'}}_{k, m_{k-1}}$ is a point of $\cT_k$ with distribution $\mu|_{\cT_k}$; moreover, $\rho_k$ is another independent point of the same distribution, by \eqref{def: rho}. We can then apply the statements for the case $k=1$ to $\tilde\chi(a_k, \rho_k; \cT_k, \cM_k)$ and find that (i) $\bE[N^{a_k, \rho_k}]<\infty$; (ii) conditional on their masses, $C^{_{a_k, \rho_k}}_m, 1\le m\le N^{a_k, \rho_k}$, are independent and distributed as rescaled Brownian CRT; (iii) for each $1\le m\le N^{a_k, \rho_k}$, 
$a^{_{a_k, \rho_k}}_{\,m}$ is a $\mu$-random point restricted to $C^{_{a_k, \rho_k}}_m$ and $p^{_{a_k, \rho_k}}_{\,m}$ is either its root or another independent point. Combined with \eqref{def: chi''} and the induction hypothesis, this leads to the statements for $k$. 
\end{proof}

\begin{proof}[Proof of Lemma \ref{lem: gamma_k}]
It follows from Lemma \ref{lem: chi} that the number of  summands in \eqref{def: gamma} is finite, and each summand is bounded by the diameter of $\cT$. This then yields $\gamma_k(i, j)<\infty$, a.s.
\end{proof}

\subsection{Convergence of $(\Gamma_k)_{k\ge 1}$}
This subsection is devoted to proving the almost sure convergence of the sequence $\{\Gamma_k=(\gamma_k(i, j))_{1\le i, j<\infty}: k\ge 1\}$ in the product topology. Note that since $(\eta_i)_{i\ge 1}$ is an i.i.d.~sequence, it suffices to prove the convergence for the sequence 
$\{\gamma_k(1, 2): k\ge 1\}$. 

\subsubsection{A Markov chain representation of $(\gamma_k(1, 2))_{k\ge 1}$. }

Let 
$\cS^\downarrow=\{\mathbf x=(x_1, x_2, \dots): x_1\ge x_2\ge \cdots \ge 0\text{ and }\sum_{i\ge 1}x_i\le 1\}$. For $\mathbf x\in \cS^\downarrow$, its $\ell_1$-norm is $\|\mathbf x\|_1=\sum_i x_i$. 
Let $\cS_f^\downarrow$ be the subset of $\cS^\downarrow$ which consists of those $\bx\in\cS^\downarrow$ for which there exists some $n\in \bbN$ such that $x_i=0$ for all $i\ge n$.  

For each $k\ge 1$, recall the collection $\chi^{1, 2}_k$ from \eqref{def: chi'} and recall that $N_k:=N^{_{1, 2}}_k=|\chi^{_{1, 2}}_k|$ is finite $\bP$-a.s.~according to Lemma \ref{lem: chi}. 
Then let $\bm_k=(\mathrm m_{k, n}: n\ge 1)$ be the sequence obtained from $\{\mu(C^{_{1, 2}}_{k, m}): 1\le m\le N^{i, i'}_k\}$ by a re-ordering in decreasing order and completed with infinitely many $0$. Observe that $\bm_k\in \cS_f^\downarrow$ since $m_{k, n}=0$ for all $n>N_k$. 

\begin{prop}\label{prop: rep_gam}
For $k\ge 1$, let $\bm_k=(\mathrm m_{k, n})_{n\ge 1}$ be defined as above. 
Under $\bP$, the sequence $(\bm_k)_{k\ge 1}$ is a Markov chain taking values in $\cS_f^\downarrow$ which evolves in the following way: for each $k\ge 1$, 
\begin{compactenum}[--]
    \item with probability $1-\|\bm_k\|_1$\,, $\bm_{k+1}=\bm_k$\,, and 
    \item for $1\le n\le N_k$\,, with probability $\mathrm m_{k, n}$\,, $\bm_{k+1}$ is obtained by 
    replacing in $\bm_k$ the element $m_{k, n}$ by $m_{k, n}\cdot\widetilde \bm$, 
    where $\widetilde \bm$ is an independent copy of $\bm_1$\,, and then sorting the sequence thus obtained  in decreasing order.
\end{compactenum}
Moreover for each $k\ge 1$, there exists a sequence of positive real numbers $(\rR_{k, n})_{1\le n\le N_k}$ such that 
\begin{equation}\label{eq: gam}
\gamma_{k}(1, 2)=\sum_{1\le n\le N_k} \sqrt{m_{k, n}}\, \rR_{k, n}\,.
\end{equation}
Under $\bP$ and given that $N_k=p\in \mathbb N$,  $(\rR^{k}_n)_{1\le n\le N_k}$ consists of $p$ independent Rayleigh random variables 
which are independent of $\bm_k$\,. 
\end{prop}

\begin{proof}
By the definition \eqref{def: chi''} of $\chi^{_{1, 2}}_k$, $\bm_{k+1}=\bm_k$ iff $V_{k+1}\in \{C^{_{1, 2}}_{k,m}: 1\le m\le N_k\}$. 
We can readily check by an induction on $k$ that $\{C^{_{1, 2}}_{k, m}: 1\le m\le N_k\}$ is a sub-collection of $\{C_{1, i}: 1\le i\le N_k \}$, the closures of the connected components of $\cT\setminus\Span(\cT, \mathbf V_k)$. On the other hand, $V_{k+1}$ is a point independent of $\mathbf V_k$ with the law $\mu|_{\cT}$. Thus, under $\bP$, the event  $V_{k+1}\in \{C^{_{1, 2}}_{k,m}: 1\le m\le N_k\}$ takes place with probability $1-\sum_{1\le m\le N_k} \mu(C^{_{1, 2}}_{k,m})=1-\|\bm_k\|_1$. Next, suppose that $V_{k+1}\in C^{_{1, 2}}_{k, m_{k}}$ for some index $1\le m_{k}\le N_k$, which takes place with probability $\mu(C^{_{1, 2}}_{k, m_{k}})$. In that case, we have $\cT_{k+1}=C^{_{1, 2}}_{k, m_{k}}$. We have seen in Lemma \ref{lem: chi} that $\cT_{k+1}=C^{_{1, 2}}_{k, m_{k}}$ is a rescaled Brownian CRT and that the points $a^{_{1, 2}}_{k, m_{k}}, \rho_{k+1}$ are independent and distributed according to $\mu|_{\cT_{k+1}}$. We then deduce from \eqref{def: chi''} the distribution of $\bm_{k+1}$ in this case. 
In this way, we have checked the transition probabilities of $(\bm_k)_{k\ge 1}$. The expression \eqref{eq: gam} is a direct consequence of \eqref{def: gamma} and the statements (a-c) in Lemma \ref{lem: chi}. 
\end{proof}

\subsubsection{Polynomial decay of a self-similar fragmentation chain}
\label{sec:decay_frag_chain}

The dynamic of $(\bm_k)_{k\ge 1}$ as described in Proposition \ref{prop: rep_gam} is that of a discrete-time self-similar fragmentation chain with index of self-similarity $1$. Self-similar fragmentation chains are studied in Bertoin \cite[Chapter 1]{bertoinfrag} and a series of papers including Bertoin and Gnedin \cite{BeGn2004}. Here, we apply their results on  the asymptotic behavior of fragmentation chains in order to obtain the following. 
\begin{lem}\label{lem: shuff_1}
Let $(\bm_k)_{k\ge 1}$ be as in Proposition \ref{prop: rep_gam}. 
There exists some $\alpha\in (0, 1)$ such that 
$$
\lim_{k\to\infty} k^\alpha \|\bm_k\|_1=0, \quad \bP\text{-almost surely .}
$$
\end{lem}

The proof of Lemma \ref{lem: shuff_1} will occupy the rest of this part. Note that the Chapter 1 of \cite{bertoinfrag} studies the continuous-time versions of self-similar fragmentation chain, which can be related to the discrete-time versions by a time-change. Let us first recall some terminology from there. 

\paragraph{Continuous-time fragmentation chain. } 
Let $\varpi$ denote the law of $\bm_1$ under $\bP$.
We consider a self-similar fragmentation chain $\{\tilde{\mathbf Z}(t)=(\tilde Z_i(t)_{i\ge 1}: t\ge 0\}$ with index of 
self-similarity $1$ and dislocation measure $\varpi$ starting from the initial state 
$(1, 0, 0, \cdots)$ as defined in \cite[][Definition 1.1]{bertoinfrag}, which is a continuous-time Markov chain taking values in $\cS^\downarrow$ whose total jump rate at time $t$ is $\|\tilde{\mathbf Z}(t)\|_1\in [0, 1]$. Using standard facts about Poisson processes and the construction of fragmentation chains in \cite{bertoinfrag}, we can construct on some probability space $(\Omega, \cF, \bbP)$ the following processes:
\begin{compactitem}
\item
a Poisson process  of rate $1$ which jumps at times $\tau_1<\tau_2<\cdots<\tau_i<\cdots$;
\item
a process $({\mathbf Z}(t))_{t\ge 0}$ having the same distribution as $(\tilde{\mathbf Z}(t))_{t\ge 0}$ such that the set of discontinuities of $t\mapsto \tilde{\mathbf Z}(t)$ is a subset of $(\tau_i: i\ge 1)$.
\end{compactitem}
Then by Proposition \ref{prop: rep_gam}, we have
\begin{equation}\label{id: markov-chain}
({\mathbf Z}(\tau_k))_{k\ge 1} \text{ under }\bbP \, \eqd (\bm_k)_{k\ge 1} \text{ under }\bP\,.
\end{equation}
Note that, in particular, we have ${\mathbf Z}(\tau_1)\eqd \bm_1=(m_{1, n})_{n\ge 1}$. 

Let $p^\ast$ be the \emph{Malthusian exponent} associated with $\varpi$, namely, $p^\ast\in [0, 1]$ is such 
that \[\bE\bigg[\sum_{1\le n\le N_1} \mathrm m_{1, n}^{p^\ast}\bigg]=1\] (see \cite[Section 1.2.2]{bertoinfrag}). 
The following is a consequence of Theorem~1 of \cite{BeGn2004}.
\begin{lem}
Let $\{{\mathbf Z}(t)=(Z_i(t)_{i\ge 1}: t\ge 0\}$ be a self-similar fragmentation chain with index of 
self-similarity $1$ and dislocation measure $\varpi$ which is defined on $(\Omega, \cF, \bbP)$ as above. Then $p^\ast\in(0, 1)$ and for any $\delta\in (0, 1)$,
\begin{equation}\label{eq: dl}
\limsup_{t\to\infty}t^{\,\delta(1-p^\ast)}\sum_{i\ge 1}Z_i(t) 
=0 \quad \bbP\text{-almost surely .}
\end{equation}
\end{lem}

\begin{proof}
First, let us show that $p^\ast\in (0, 1)$. Recall $N_1$ is the number of non zero elements of $\bm_1$, which is also equal to $\sI(\eta_1, \eta_2; \cT, \cM_1)+\sI(\eta_2, \eta_1; \cT, \cM_1)-1$ by our previous definitions. Then  Lemma \ref{lem: tau} tells that $N_1<\infty$, $\bP$-a.s. On the other hand, $\bm_1$ is a sub-collection of the $\mu$-masses of the connected components of $\cT\!\setminus\!\Span(\cT, \{V_1\})$. Therefore, we must have $\|\bm_1\|_1<1$, $\bP$-a.s. This shows $p^\ast<1$. To see why $p^\ast>0$, note that $N_1=1$ if and only if $\eta_1$ and $\eta_2$ are found in the same component of $\cT\!\setminus\!\Span(\cT, \{V_1\})$, which occurs with probability strictly smaller than $1$. Then on the event $N_1\ge 2$, we have $\lim_{p\to 0+}\sum_{1\le n\le N_1}\mathrm m_{1,n}^p\ge 2$. This shows $p^\ast>0$. 

We introduce
\begin{equation}
\label{eq: martingale}
Y(t):=\sum_{i\ge 1}Z^{p^\ast}_i(t), 
\end{equation}
which is a strictly positive martingale by the choice of $p^\ast$. Denote by $Y(\infty)\in [0, \infty)$ its almost sure limit. Next, we check that the hypothesis of Theorem~1 in \cite{BeGn2004} are fulfilled:   with the notation there, we see that $\beta^\ast=p^\ast$, $\beta_a\le 0$ (since $\bE[N_1]<\infty$ by Lemma \ref{lem: tau}), $\sigma=\varpi$ is diffuse and the conditions (1) on page 577 all hold for $\varpi$. Then, by the above mentioned theorem, 
for every $k\ge 1$, there exists a constant $C_k\in (0, \infty)$ such that
$$
\sup_{t\ge 0}t^k\cdot \bbE \Bigg[\sum_{i\ge 1}Z^{p^\ast+k}_i(t)\Bigg] \le C_k,
$$
from which it follows immediately that $\sup_{t\ge 0}t^k\cdot \bbE Z^{p^\ast+k}_1(t)\le C_k$.
Now, for any $\delta\in (0, 1)$ and $\ep>0$, by Markov's inequality, we obtain at time $m> 0$,
\begin{equation}\label{eq:sup_markov}
\bbP\big(Z_1(m)\ge \ep m^{-\delta}\big)
\le \frac{ m^k\cdot \bbE [Z^{p^\ast+k}_1(m)]}{\ep^{p^\ast+k}m^{(1-\delta)k-\delta p^\ast}}
\le \frac{C_k}{\ep^{p^\ast+k}m^{(1-\delta)k-\delta p^\ast}}.
\end{equation}
Choosing $k$ large enough so that $k> (1+\delta p^\ast)/(1-\delta)$, one sees that \eqref{eq:sup_markov} 
implies that $\sum_{m\ge 1}\bbP(Z_1(m)\ge \ep m^{-\delta})<\infty$ and 
$\limsup_{m\to\infty}m^\delta Z_1(m) \le \ep$ almost surely, by the Borel--Cantelli lemma.
As $\ep$ was chosen arbitrary, we then obtain 
that $m^\delta Z_1(m)\to 0$ a.s.\ as $m\to\infty$, and $t^{\delta} Z_1(t) \to 0$ a.s.\ as $t\to\infty$ as well by monotonicity.
Now note that for any $t\ge 0$, 
$$
\sum_{i\ge 1}Z_i(t)=
\sum_{i\ge 1}Z^{1-p^\ast}_i(t)\cdot Z^{p^\ast}_i(t)
\le Z^{1-p^\ast}_1(t)\cdot Y(t).
$$
Then, for any $\delta\in (0, 1)$,
\[
\limsup_{t\to\infty}t^{\delta(1-p^\ast)}\sum_{i\ge 1}Z_i(t) 
\le Y(\infty)\cdot \limsup_{t\to\infty} (t^{\delta}Z_1(t))^{1-p^\ast}=0, 
\]
almost surely, since $p^\ast\in (0,1)$. This completes the proof of the lemma.
\end{proof}

\begin{proof}[Proof of Lemma \ref{lem: shuff_1}]
We work on a probability space where the equality in \eqref{id: markov-chain} holds almost surely. 
By the strong law of large numbers, we have $\tau_k/k\to 1$ almost surely as $k\to\infty$. 
Therefore, we obtain that for any $\delta\in (0, 1)$, 
\[
\limsup_{k\to\infty} k^{\,\delta(1-p^\ast)} \|\bm_k\|_1 
= \limsup_{k\to\infty} k^{\,\delta(1-p^\ast)}\sum_{i\ge 1}Z_i(\tau_k)=0, \quad \text{ almost surely,}
\]
by \eqref{eq: dl}. This  proves Lemma~\ref{lem: shuff_1} by taking $\alpha=\delta(1-p^\ast)$.
\end{proof}

\subsubsection{Concentration around the conditional expectations}
\label{sec:concentration_rayleigh}

In this part, we rely on Lemma \ref{lem: shuff_1} and the exponential tail of a Rayleigh random variable to show the following result. 
\begin{lem}\label{lem: shuff_2}
$\bP$-almost surely, $\gamma_k(1,2) - \bE[\gamma_k(1,2)\, | \, \bm_k]\to 0$ as $k\to \infty$.
\end{lem}

Let $\rR$ be a Rayleigh random variable defined on some probability space $(\Omega, \cF, \bbP)$, namely, $\rR$ has density $x e^{-x^2/2}\mathbf{1}_{\{x>0\}}$. 
Then, one readily verifies that $\rR-\Ec{\rR}$ is \emph{sub-Gaussian} in the sense that 
there exists a constant $v$ such that for every $\lambda\in \bbR$, one has 
\[
\log \Ec{e^{\lambda (\rR- \Ec{\rR})}}\le \tfrac{1}{2}\,\lambda^2 v.
\]
(See \cite[][Theorem~2.1, p.\ 25]{BoLuMa2012a}.)
We may thus apply concentration results for sub-Gaussian random variables such as the ones 
presented in Section~2.3 of \cite{BoLuMa2012a}. To that end, we set for each $k\ge 1$, 
\[
\sigma_k:=\gamma_k(1, 2)-\bE[\gamma_k(1, 2)\,| \,\bm_k]=\sum_{1\le n\le N_k}\sqrt{\mathrm m_{k, n}}\, \big(\rR_{k, n}-\bE[\rR_{k, n}]\big),
\]
by \eqref{eq: gam}, where according to Proposition~\ref{prop: rep_gam}, conditional on $\bm_k$, $(\rR_{k, n}, 1\le n\le N_k)$ are  i.i.d.\ copies of $\rR$. 
Therefore, by the above mentioned concentration results, we find 
\begin{equation}\label{eq: exp_tail}
\bP(|\sigma_k|\ge \ep \,|\,\bm_k)\le 2\exp\Big(-\frac{\ep ^2}{2v\|\bm_k\|_1}\Big), \quad \forall\, \ep>0.
\end{equation}
If $(A_k, B_k, C_k)_{k\ge 1}$ are sequence of events satisfying that $A_k\subset B_k\cup C_k$ 
for each $k\ge 1$, then it is elementary that 
$\bP(\limsup_k A_k)\le \bP(\limsup_k B_k)+\bP(\limsup_k C_k)$. Here, we take 
$$
A_k=\{|\sigma_k|\ge \ep\}, \quad 
B_k=A_k\cap \{\|\bm_k\|_1\le k^{-\alpha}\}, \quad 
C_k=\{\|\bm_k\|_1> k^{-\alpha}\}
$$
with the same $\alpha$ as in Lemma~\ref{lem: shuff_1}. Then $\bP(\limsup_k C_k)=0$ by 
Lemma~\ref{lem: shuff_1}. On the other hand, we deduce from \eqref{eq: exp_tail} that
$$
\sum_{k\ge 1}\bP(B_k)= 
\sum_{k\ge 1}\bE\Big[\bP\big(|\sigma_k |\ge \ep \,|\, \bm_k\big)\cdot \boldsymbol{1}_{\{ \|\bm_k\|_1\le k^{-\alpha}\}}\Big]
\le \sum_{k\ge 1}2e^{-\ep^2k^\alpha/(2v)}<\infty,
$$
which entails that $\bP(\limsup_k B_k)=0$ by the Borel--Cantelli lemma. 
Hence, $\bP(\limsup_k A_k)=0$, which means $\limsup_k |\sigma_k|< \ep $ almost surely. 
Since $\ep>0$ was arbitrary, the proof of Lemma~\ref{lem: shuff_2} is now complete.

\subsubsection{A coupling via partial cut trees}
\label{sec: coupling}

Thanks to Lemma~\ref{lem: shuff_2}, the last step to show the convergence of $(\gamma_k(1, 2))_{k\ge 1}$ consists in proving that under~$\bP$, $\bE[\gamma_k(1,2)\, |\, \bm_k]$ converges almost surely as $k\to\infty$. For this, we rely on a coupling of $(\bm_k)_{k\ge 1}$ with a sequence of masses defined on the partial cut trees. 

Let $k\ge 1$. We recall the following notation in Lemma \ref{lem: key}: $\eta, \eta'$ are two independent points of $\cT$ with distribution $\mu|_{\cT}$; the collection $\{\Delta^{_{(i)}}_m: (i, m)\in \cI_k(\eta, \eta')\}$ consists of the subsets of $\cG_k$ (the $k$-partial cut tree of $\bcT$) which intersect the geodesic $\llb\eta,\eta'\rrb_{\cT}$. Denote $N'_k=|\cI_k(\eta, \eta')|$. 
Let $\{(\mathrm m'_{k, n}, \rR'_{k, n}): 1\le n\le N'_k\}$ be the sequence obtained from $\{(\mu(\Delta_m^{_{(i)}}), \rR_{i, m}): (i, m)\in \cI_k(\eta, \eta')\}$ by arranging the first coordinates in decreasing order.  
It follows from Lemma \ref{lem: key} that 
\begin{equation}\label{eq: gam'}
D\!:=d(\eta, \eta')=\sum_{1\le n\le N'_k}\sqrt{\mathrm m'_{k, n}}\, \rR'_{k, n}\,,
\end{equation}
where, given $N'_k=p\in \bbN$, $(\rR'_{k, n}: 1\le n\le N'_k)$ are $p$ independent Rayleigh random variables which are independent of $\bm'_k:=(\mathrm m'_{k, n})_{n\ge 1}$ with $m'_{k, n}=0$ for $n>N'_k$. 
 
\begin{lem}\label{lem: rep_delta}
Under $\bP$, $(\bm'_k)_{k\ge 1}$ has the same distribution as  the Markov chain $(\bm_k)_{k\ge 1}$. 
\end{lem}

\begin{proof}
We first show that $\bm_1\eqd\bm'_1$. Recall from \eqref{def: cM_k} the collection $\cM_1$.  
For $i=1, 2$, recall from \eqref{eq: mk_1} the definition of $\chi^{\eta_i}(\cT, \cM_1)$. Note that it tells that the sequence $(\mu(C^{\eta_i}_j): j\ge 1)$ has the following distribution. Conditional on $(\bcT, V_1)$, $\mu(C^{\eta_i}_1)$ is a random element of $\{\mu(C_{1,m}): m\in \N\}$ chosen by size-biased sampling, that is, for $m\in \N$, 
\[
\bP\big(\mu(C^{\eta_i}_1) =\mu(C_{1, m}) \,|\, \bcT, V_1\big)=\bP(\eta_i\in C_{1, m})=\mu(C_{1, m}),
\]
since $\bP$-a.s, the $\mu$-masses of $C_{1, m}, m\in\N$, are all distinct. More generally for $j\ge 1$, given $d(\rho, p^{\eta_i}_{j-1})$, $\mu(C^{\eta_i}_j)$ is chosen from $\{\mu(C_{1, m}): m\in \N, d(\rho, \Rt(C_{1, m}))>d(\rho, p^{\eta_i}_{j-1})\}$ by size-biased sampling. Moreover, conditional on $\bcT$ and $V_1$, the two sequences $(\mu(C^{\eta_1}_j): j\in \mathbb Z_+)$ and $(\mu(C^{\eta_2}_j): j\in \mathbb Z_+)$ are independent, since $\eta_1$ and $\eta_2$ are independent. 
Note that we also have $\sI(\eta_1, \eta_2; \cT, \cM_1)=\inf\{j\in \N: \exists\, j'\in\N \text{ s.t. }\mu(C^{_{\eta_1}}_j)=\mu(C^{_{\eta_2}}_{j'})\}$, by the fact that $\mu(C_{1, m}), m\ge 1$, are a.s.~distinct. It follows that we can write $\bm_1=F(\bcT, V_1; \eta_1, \eta_2)$ for some measurable function $F$, outside a $\bP$-null set. 
Next, we show that a.s.~$\bm'_1=F(\cG_1, V'_1; \eta, \eta')$. But this is a consequence of Theorem 5.1 in \cite{ABH10}. Indeed, recall from Proposition \ref{prop: cG_k} the spinal decomposition $\D(\cG_1, \{V'_1\})$. 
Denote by $\{\tilde\Delta_{m_j}: j\in \mathbb Z_+\}$ the sequence of those $\Delta^{_{(1)}}_m, m\ge 1$, which satisfy $\Delta^{_{(1)}}_m\cap\llb \eta, V_1\rrb\ne \varnothing$ such that  $h^{_{(1)}}_{m_1}<h^{_{(1)}}_{m_2}<\cdots$. Then the above mentioned theorem identifies the distribution of $\{\tilde\Delta_{m_j}: j\in \mathbb Z_+\}$ as that of $\{\mu(C^{\eta_1}_j): j\in \N\}$. This then entails that $\bm'_1$ can be written as $F(\cG_1, V'_1; \eta, \eta')$ with the same $F$ as before. Then we have $\bm_1\eqd \bm'_1$, since $(\cG_1, V'_1)\eqd (\cT, V_1)$ by Proposition \ref{prop: cG_k}. 

The rest of the proof is very similar to that of Proposition \ref{prop: rep_gam}. The main differences lie in that we use Corollary \ref{cor: 1} and Lemma \ref{lem: rec} instead of Lemma \ref{lem: chi} and \eqref{def: chi''}. We omit the details. 
\end{proof}

As $(\bm'_k)_{k\ge 1}$ has the same distribution as $(\bm_k)_{k\ge 1}$, Lemma~\ref{lem: shuff_1} 
also holds for $(\bm'_k)_{k\ge 1}$. Furthermore, combined with Lemma \ref{lem: rep_delta}, the concentration arguments already used in the 
course of the proof of Lemma~\ref{lem: shuff_2} imply that, a.s.,
$$
D-\bE[D\,|\, \bm_k']
= \sum_{1\le n\le N'_k}\sqrt{\mathrm m'_{k, n}} \cdot\big(\rR'_{k, n}-\bE[\rR'_{k, n}]\big)
\xrightarrow{k\to\infty} 0, \quad \bP\text{-a.s.}~.
$$
This entails that  
\[
\bE[D\,|\,\bm_k']=\bE[\rR'_{1, 1}] \cdot\sum_{1\le n\le N'_k} \sqrt{\mathrm m'_{k, n}}\, , \quad k\ge 1, 
\]
converges almost surely to $D$. Since the sequence $(\bE[\gamma_k(1,2)\,|\, \bm_k])_{k\ge 1}$ has the 
same distribution by Lemma \ref{lem: rep_delta}, it also converges almost surely to some random variable that
has the same distribution as $D$. 
Combined with Lemma \ref{lem: shuff_2}, we have shown the following. 

\begin{prop}\label{lem: shuff_3}
Under $\bP$, there exists a Rayleigh random variable $\gamma_\infty(1, 2)$ such that $\gamma_k(1, 2)\to \gamma_\infty(1, 2)$ almost surely. 
\end{prop}

\subsection{Proof of Theorem \ref{thm: cv_gamma}}

By exchangeability and by Proposition \ref{lem: shuff_3}, for all $i, i'\ge 1$, $i\ne i'$, there exists a Rayleigh random variable $\gamma_\infty(i, j)$ such that $\lim_{k\to\infty}\gamma_k(i, j)=\gamma_\infty(i, j)$ $\bP$-almost surely. This proves the point a) of Theorem \ref{thm: cv_gamma}. For the rest of the statement, let us begin with the proof of \eqref{id: matrix}. 

We start with a distributional identity for $\Gamma_k$. 
Let $(\xi_i)_{i\ge 1}$ be an independent sequence of i.i.d.~points of $\cT$ with common distribution $\mu|_{\cT}$ and set $\xi_0=\rho$. Then Equation (6.7) and Equation (6.1) in \cite{BW} entail that
\begin{equation}\label{id: matrix_k}
\big(d(\eta_i, \eta_j), \gamma_k(i+1, j+1)\big)_{ 0\le i, j<\infty}\text{ under } \bP \;\eqd \;\big(d_{\cG_k}(\xi_i, \xi_j), d(\xi_{i+1}, \xi_{j+1})\big)_{0\le i, j<\infty}  \text{ under } \bP, 
\end{equation}
in the case where $k=1$. However, the arguments in \cite[Section 6]{BW} can be readily adapted to a general proof for $k\ge1$. Therefore, \eqref{id: matrix_k} holds for any $k\ge 1$. On the other hand, for $k\ge 1$, recall from Proposition \ref{prop: cG_k} that conditional on $\mu_{\cG_k}$, the points $V'_i, 1\le i\le k$, of $\cG_k$ are distributed as $k$ independent points with common distribution $\mu_{\cG_k}$. Then,
\begin{equation}\label{id'}
\Big(\big(d_{\cG_k}(\xi_i, \xi_j)\big)_{0\le i, j\le k}\,,  \big(d(\xi_{i}, \xi_{j})\big)_{ i, j\ge 0}\Big)\eqd \Big(d_{\cG_k}(V'_i, V'_j)_{0\le i, j\le k}\,, \big(d(V_{i}, V_{j})\big)_{ i, j\ge 0}\Big),
\end{equation}
Now let $n\ge 1$ and let $G, H: \bbR^{\mathbb Z_+\times\mathbb{Z}_+} \to \bbR$ be two continuous bounded functions with respect to the product topology which are supported on $\bbR^{n\times n}$. 
Set $V'_0$ to be the root of $\cG_k$. We have used the same sequence $(V'_i)_{i\ge 0}$ for different $\cG_k$; this will not cause confusion, since by Proposition \ref{prop: cG_k}, $\Span(\cG_k; \{V'_1, \dots, V'_k\})$ is isometric to $\Span(\cG; \{V'_1, \dots, V'_k\})$. In particular, we have the fact that $d_{\cG_k}(V'_i, V'_j)=\delta(i, j)$, $0\le i, j\le k$, for all $k$. 
We then obtain from the convergence of $(\Gamma_k)_{k\ge 1}$, equations \eqref{id: matrix_k}, \eqref{id'} and this fact  that
\begin{align*}
\bE\Big[G\big((d(\eta_i, \eta_j))_{0\le i, j<\infty}\big)&H\big( (\gamma_\infty(i, j))_{1\le i, j<\infty}\big)\Big]\\
&=\lim_{k\to\infty} \bE\Big[G\big((d(\eta_i, \eta_j))_{0\le i, j<\infty}\big)H\big( (\gamma_k(i, j))_{1\le i, j<\infty}\big)\Big]\\
&= \bE\Big[G\big((\delta(i, j))_{0\le i, j<\infty}\big)H\big( (d(V_i, V_j))_{1\le i, j<\infty}\big)\Big],
\end{align*}
which proves \eqref{id: matrix}, as $n$ is arbitrary. 

Next, 
we follow Aldous \cite{aldcrt3} to construct the measured real tree $\shuff(\cT)$; see also \cite[Section 1.4]{Bert12}. First, we observe that we can construct a family of rooted real trees $(\mathcal R_k, d_{\cH}), k\ge 1$, such that 1) $\mathcal R_1\subseteq\mathcal R_2\subseteq\cdots$ as metric spaces; 2) each $\mathcal R_k$ has exactly $k+1$ leaves which we denote as $\varsigma_0, \varsigma_1, \dots, \varsigma_k$ and a common root $\rho_{\cH}=\varsigma_0$; 3) for each $k\ge 1$, the distance between $\varsigma_i$ and $\varsigma_j$ is given by $\gamma_\infty(i+1, j+1)$, $0\le i, j\le k$. On the other hand, note that \eqref{id: matrix} entail that $(\gamma_\infty(i, j))_{1\le i, j<\infty}$ under $\bP$  has the same distribution as $(d(V_i, V_j))_{1\le i, j<\infty}$. 
Since $\bcT$ satisfied the so-called \emph{leaf-tight property}, we deduce this also holds for $(\mathcal R_k)_{k\ge 1}$, namely, $\inf_{j\ge 1}d_{\cH}(\varsigma_0, \varsigma_j)=0$, $\bP$-a.s. 
Note that all this still holds if we have first conditioned on $\bcT$. Then by Theorem 3 in \cite{aldcrt3}, given $\bcT$, $(\mathcal R_k)_{k\ge 1}$ under $\bP$ allows for a representation as a measured real tree, which we denote as $\shuff(\cT)=(\cH, d_{\cH}, \mu_{\cH}, \rho_{\cH})$. Moreover, if $(\zeta_i)_{i\ge 1}$ is a sequence of i.i.d.~points of $\cH$ with common distribution $\mu_{\cH}$ and $\zeta_0=\rho_\cH$, then $(d_{\cH}(\zeta_i, \zeta_j))_{0\le i, j<\infty}$ has the same distribution as $(d_{\cH}(\varsigma_i, \varsigma_j))_{0\le i, j<\infty}$. For this reason, we can then view $(\varsigma_i)_{i\ge 1}$ as an i.i.d.~sequence of common law $\mu_{\cH}$. As $\Gamma_\infty$ characterizes $\shuff(\cT)$, also remark that the root of $\cT$ is a $\mu$-random point conditional on $\cut(\cT)$, then \eqref{id: matrix} entails that $(\bcT, \shuff(\cT))\eqd (\cut(\cT), \bcT)$.  
The proof of Theorem \ref{thm: cv_gamma} is now complete.

%%%%%%%%%%%%%%%%%%%%%%%%%%%%%%%%%%%%%%%%%%%%%%%%%%%%%%%

%\newpage

{\small
\setlength{\bibsep}{.2em}
\bibliographystyle{abbrvnat}
\bibliography{refs}
}

%%%%%%%%%%%%%%%%%%%%%%%%%%%%%%%%%%%%%%%%%%%%%%%%%%%%%%%
\end{document}